\documentclass[preprint,11pt,3p]{elsarticle}
\usepackage{amsmath,amsthm,amsfonts,amssymb}
\usepackage{xcolor}
\usepackage{algorithm}
\usepackage{multirow}
\usepackage{algpseudocode}
\usepackage{blkarray}
\usepackage{booktabs}

\usepackage{mathrsfs}
\usepackage{sectsty}
\definecolor{astral}{RGB}{46,116,181}
\sectionfont{\color{astral}}
\newtheorem{theorem}{Theorem}[section]
\newtheorem{lemma}[theorem]{Lemma}

\newtheorem{example}[theorem]{Example}
\newtheorem{remark}[theorem]{Remark}
\usepackage{hyperref}
\usepackage{subfigure}

\usepackage{tikz,xcolor}
\definecolor{lime}{HTML}{A6CE39}
\definecolor{lightblue}{rgb}{0.0, 0.0, 0.5}
\DeclareRobustCommand{\orcidicon}{%
	\begin{tikzpicture}
	\draw[lime, fill=lime] (0,0)
	circle [radius=0.16]
	node[white] {{\fontfamily{qag}\selectfont \tiny ID}};
	\draw[white, fill=white] (-0.0625,0.095)
	circle [radius=0.007];
	\end{tikzpicture}
	\hspace{-2mm}
}
\foreach \x in {A, ..., Z}{%
	\expandafter\xdef\csname orcid\x\endcsname{\noexpand\href{https://orcid.org/\csname orcidauthor\x\endcsname}{\noexpand\orcidicon}}
}

\newcommand\norm[1]{\left\lVert#1\right\rVert}

\newcommand{\mLabel}[1]{\mbox{$\scriptstyle{#1}$}}
\def \diag{\mathrm{diag}}
\def \R{{\mathbb R}}
\def \C{{\mathbb C}}
\def \rank{\mathrm{rank}}

\def \QR{\mathbb{Q}_\mathbb{R}}
\def \i{\textit{\textbf{i}}}
\def \j{\textit{\textbf{j}}}
\def \k{\textit{\textbf{k}}}

\begin{document}
\begin{frontmatter}
\title{Solution Methods and Perturbation Analysis for the Equality-Constrained Total Least Squares Problem over Reduced Biquaternions
}
\author{Neha Bhadala$^\dagger$$^a$}

\address{
 $^{\dagger}$ Department of Computational and Data Sciences, Indian Institute of Science, Bengaluru,  India
\\\textit{E-mail\,$^a$}: \texttt{bhadalaneha@gmail.com}
}

\begin{abstract}
This paper proposes a theoretical framework to address the reduced biquaternion equality-constrained total least squares (RBTLSE) problem. The objective is to find an approximate solution to the system $AX \approx B$, subject to linear constraints $CX = D$, while explicitly accounting for errors in both the coefficient matrix $A$ and the observation matrix $B$. We establish conditions under which real and complex solutions exist and develop corresponding solution methods using real and complex representations of reduced biquaternion matrices. To assess the sensitivity of the solutions to data perturbations, we derive relative normwise condition numbers and provide tight upper bounds on the relative forward errors. These results ensure the computational reliability of the proposed framework. Extensive numerical experiments validate the theoretical findings and demonstrate that the RBTLSE approach significantly outperforms the conventional reduced biquaternion equality-constrained least squares (RBLSE) method, particularly in noisy environments affecting both sides of the system.
\end{abstract}

\end{frontmatter}

\section{Introduction}\label{sec1}

Quaternions, first formulated by W.R. Hamilton in $1843$ \cite{MR237284}, extend complex numbers into a non-commutative algebra and have found broad application in fields like signal analysis, image reconstruction, and neural network models \cite{MR4835709, said2008fast, zeng2016color}. Building upon this structure, Schütte and Wenzel introduced the concept of the reduced biquaternion (RB) in 1990 \cite{schutte1990hypercomplex}. Unlike quaternions, reduced biquaternions satisfy the commutative law of multiplication, which not only simplifies their arithmetic operations but also significantly reduces computational complexity \cite{MR2087700, MR2517402}. These features make RBs particularly suitable for large-scale problems in applied fields such as color image processing, signal processing, and neural networks \cite{gai2021reduced, gai2024regularization, MR3300514, MR2087700}.

Given their algebraic simplicity and strong representational power, reduced biquaternion matrices have received increasing attention in the literature. Recent advances include the exploration of the H-representation of RB matrices for symmetric solutions \cite{han2025h}, singular value decomposition and its applications in image restoration \cite{MR4839554}, and necessary and sufficient conditions for matrix diagonalization \cite{MR4685147}. Further studies have focused on the eigenproblem of RB matrices \cite{MR4649293}, outer inverses \cite{bhadala2024outer}, and L-structured solutions for generalized RB matrix equations \cite{ahmad2025structure}.

In numerical linear algebra, the least squares (LS) formulation serves as a widely used approach for estimating solutions to overdetermined systems of the form $AX \approx B$. A common extension of LS is the equality-constrained least squares (LSE) problem, which incorporates additional linear constraints on the solution. However, a critical assumption in both LS and LSE problems is that the matrix $A$ is error-free. In real-world scenarios, this assumption is often violated due to measurement or modeling noise. The total least squares (TLS) approach was introduced to handle inaccuracies present in both the matrices $A$ and $B$. When equality constraints on solutions are present, the problem becomes an equality-constrained total least squares (TLSE) problem.

In the quaternion domain, several works have addressed LSE problems using decomposition techniques such as quaternion SVD \cite{MR4287902} and QR decomposition \cite{MR4509107}. In \cite{MR3416387}, the authors studied the special solution of quaternion TLS problem. The extension of these problems to the reduced biquaternion setting is relatively recent. The LS and LSE problems for RB matrices have been studied in \cite{ahmad2024solutions, MR4085494, MR4532609}, and the TLS problem in \cite{MR4848406}. However, the TLSE problem for reduced biquaternions remains unexplored in the existing literature. Despite growing interest in reduced biquaternion matrix analysis, no comprehensive study has yet addressed the development of a theoretical and computational framework for finding solution to the reduced biquaternion equality-constrained total least squares (RBTLSE) problem. This gap is significant, as many real-world applications involve noise or uncertainties not only in the observation matrix but also in the coefficient matrix, making such a framework essential for accurate modeling and solution.

In practical scenarios, ensuring the reliability and robustness of solutions is just as important as obtaining them. A key challenge lies in understanding how small perturbations in the input data affect the computed solution. This sensitivity analysis is crucial for evaluating the numerical stability of any proposed method.

Motivated by these research gaps, this paper aims to address the following key questions:
\begin{enumerate}
	\item Can a systematic approach be developed to compute special solutions of the RBTLSE problem?
	\item How sensitive is the RBTLSE solution to perturbations in the input matrices?
\end{enumerate}
To investigate the second question, we analyze the relative normwise condition number, which measures the extent to which a small relative perturbation in the input data can lead to a relative change in the solution. A larger condition number indicates that the problem is ill-conditioned and the solution may be highly sensitive to input errors. Furthermore, we derive tight upper bounds on the relative forward error, providing a theoretical estimate of the deviation between the computed and true solutions in the presence of perturbations. These measures offer valuable insights into the numerical reliability of the RBTLSE solution. Based on this motivation and analysis, the main contributions of this research include:
\begin{enumerate}
    \item Establishing necessary and sufficient conditions for the existence and uniqueness of real and complex solutions to the RBTLSE problem.
    \item Designing a unified and systematic solution framework based on real and complex representations of RB matrices.
    \item Formulating relative normwise condition numbers to rigorously quantify the sensitivity of both real and complex RBTLSE solutions.
    \item Deriving sharp theoretical upper bounds on relative forward errors to ensure computational stability and reliability.
    \item Demonstrating the effectiveness and robustness of the proposed framework through extensive numerical experiments.
\end{enumerate}
The proposed research provides the first unified framework for analyzing and solving the RBTLSE problem. In contrast to the RBLSE approach, the RBTLSE formulation offers enhanced accuracy in scenarios involving perturbations in both the coefficient matrix and the observation matrix. By addressing error sensitivity and structural solution methods, the proposed methodology enhances the current landscape of reduced biquaternion-based computation. It also lays the groundwork for more accurate and robust applications in areas such as color image and signal processing, where reduced biquaternion representations are naturally suited.

The remainder of the paper is structured as follows. Section~\ref{sec2} introduces the necessary notations and preliminaries. In Section~\ref{sec3}, we develop the algebraic approach for computing the real solution of the RBTLSE problem. Section~\ref{sec4} extends this methodology to derive the complex solution. In Section~\ref{sec5}, we perform a detailed perturbation analysis to assess the sensitivity of the solutions. Section~\ref{sec6} presents numerical examples that validate the theoretical results and illustrate the effectiveness of the proposed methods.

\section{Notation and Preliminaries}\label{sec2}
\subsection{Notation}
In this work, we use \( \R \), \( \C \), and \( \QR \) to represent the domains of real, complex, and RB numbers, respectively. The corresponding matrix spaces of dimension \( m \times n \) over these sets are represented by \( \R^{m \times n} \), \( \C^{m \times n} \), and \( \QR^{m \times n} \). We use $\norm{\cdot}_F$ for the Frobenius norm and $\norm{\cdot}_2$ for the spectral norm. Standard matrix operations are indicated as follows: for $A \in \C^{m \times n}$, $A^T$ is the transpose, $A^H$ the conjugate transpose, and $A^{-1}$ the inverse. The row space and null space of \( A \) are denoted by \( \mathcal{R}(A) \) and \( \mathcal{N}(A) \), respectively. The matrix $\Pi_{(d,n)} \in \mathbb{R}^{dn \times dn}$ is defined as $\Pi_{(d,n)} = \sum_{i=1}^{d} \sum_{j=1}^{n} E_{ij} \otimes E_{ij}^T$, where $E_{ij}$ is the elementary matrix with a $1$ at position $(i,j)$. In numerical experiments, we use Matlab's \texttt{rand} for uniform random values and \texttt{randn} for normal random values.
\subsection{Preliminaries}
A reduced biquaternion admits a unique representation in the form:  $\alpha= \alpha_{0}+ \alpha_{1}\i+ \alpha_{2}\j+ \alpha_{3}\k$,
where $\alpha_{0}, \alpha_1, \alpha_2, \alpha_3 \in \R$. The basis elements satisfy the following multiplication rules: $\j^2= 1$, $\i^2=  \k^2= -1$,
$\i\j= \j\i= \k$, $\j\k= \k\j= \i$, $\k\i= \i\k= -\j$. Using this structure, $\alpha$ can also be represented as $\alpha=\beta_1 +\beta_2\j$, where $\beta_1=\alpha_0 + \alpha_1 \i \in \C$ and $\beta_2=\alpha_2 + \alpha_3\i \in \C$. 

The norm of $\alpha$ is given by $\norm{\alpha}= \sqrt{\alpha_0^2+ \alpha_1^2+ \alpha_2^2+ \alpha_3^2}$, which is equivalently expressed as $\sqrt{|\beta_1|^2+|\beta_2|^2}$. For a matrix $P = (p_{ij}) \in \QR^{m \times n}$, the Frobenius norm is characterized according to \cite{MR4137050} as:
\begin{equation}\label{eq2.1}
	\norm{P}_F= \sqrt{\sum_{i= 1}^{m} \sum_{j= 1}^{n} \norm{p_{ij}}^{2}}.
\end{equation}
Let $P=P_0+ P_1\i+ P_2\j+ P_3\k= R_1+R_2\j \in \QR^{m \times n}$, where $P_t \in \R^{m \times n}$ for $t=0,1,2,3$ and $R_t \in \C^{m \times n}$ for $t=1,2$. The matrix $P$ admits two fundamental representations—real and complex—denoted as $P^R$ and $P^C$, respectively. These representations are structured as follows:
\begin{equation}\label{eq2.2}
	P^R= \begin{bmatrix}
		P_0  &  -P_1  &  P_2  &  -P_3 \\
		P_1   &   P_0  &  P_3  &   P_2  \\
		P_2   &  -P_3 & P_0  &  -P_1  \\
		P_3  &  P_2  &  P_1  &  P_0 
	\end{bmatrix}, ~ P^C= \begin{bmatrix}
		R_1 & R_2 \\
		R_2 & R_1 
	\end{bmatrix}.
\end{equation}
Several fundamental properties hold for RB matrices. Given $\zeta \in \R$, $P, Q \in \QR^{m \times n}$, and $T \in \QR^{n \times t}$, we have $P= Q \iff P^R= Q^R$, $(P+Q)^R= P^R+ Q^R$, $(\zeta P)^R= \zeta P^R$, and $(PT)^R = P^R T^R$. Define the transformation matrices:
\begin{equation*}
	\mathcal{K}_m= \begin{bmatrix}
		0      &   -I_m   &  0     & 0    \\
		I_m  &    0        &  0     & 0    \\
		0      &    0        &  0     & -I_m \\
		0      &    0        &  I_m & 0
	\end{bmatrix}, ~
	\mathcal{L}_m= \begin{bmatrix}
		0      &   0      &  I_m   & 0           \\
		0      &    0     &  0       &   I_m     \\
		I_m  &    0     &  0       & 0           \\
		0      &    I_m &  0       & 0
	\end{bmatrix}, ~
	\mathcal{M}_m= \begin{bmatrix}
		0      &   0      &  0       & -I_m           \\
		0      &    0     &  I_m    &  0    \\
		0      &  -I_m  &  0       & 0           \\
		I_m   &    0     &  0       & 0
	\end{bmatrix}.
\end{equation*}
Define $P_c^R= [P_0^T, P_1^T, P_2^T, P_3^T]^T$ as the leading block column of the matrix $P^R$. Then, the matrix $P^R$ can be expressed as
\begin{equation}\label{eq2.3}
	P^R= [P_c^R, \mathcal{K}_m P_c^R, \mathcal{L}_m P_c^R, \mathcal{M}_m P_c^R].  
\end{equation}

For $\zeta \in \C$, $P, Q \in \QR^{m \times n}$, and $T \in \QR^{n \times t}$, we have $P= Q \iff P^C= Q^C$, $(P+Q)^C= P^C+ Q^C$, $(\zeta P)^C= \zeta P^C$, and $(PT)^C = P^C T^C$. We define
\begin{equation*}
	\mathcal{N}_m= \begin{bmatrix}
		0 & I_m \\
		I_m & 0
	\end{bmatrix}.
\end{equation*}
Define $P_c^C= [R_1^T, R_2^T]^T$ as the leading block column of the matrix $P^C$. Then, the matrix $P^C$ can be expressed as
\begin{equation}\label{eq2.4}
	P^C= [P_c^C, \mathcal{N}_m P_c^C].  
\end{equation}
The real and complex representations provide the mathematical foundation for defining the Frobenius norm within the RB matrix framework. As established in fundamental studies \cite{MR4085494, MR4532609}, this formulation maintains coherence with equation \eqref{eq2.1}, ensuring theoretical consistency throughout the analysis. For any matrix $P \in \QR^{m \times n}$, the Frobenius norm can be expressed through either representation via the following relationship:
\begin{equation}\label{eq2.211}
	\norm{P}_F= \frac{1}{2} \norm{P^R}_F = \norm{P^R_c}_F= \frac{1}{\sqrt{2}} \norm{P^C}_F=\norm{P^C_c}_F.
\end{equation}
From equations \eqref{eq2.1}, \eqref{eq2.2}, and \eqref{eq2.211}, we derive the following important lemmas.
\begin{lemma}\label{lem2}
	Let $P= P_0+ P_1\i+ P_2\j+ P_3\k \in \QR^{m \times n}$ and $Q= Q_0+ Q_1\i+ Q_2\j+ Q_3\k \in \QR^{m \times d}$. Then, the following results hold.
	\begin{enumerate}
		 \item[\rm (i)] $\|[P, Q]\|_F= \frac{1}{2}\|[P, Q]^R\|_F$.
		 \item[\rm (ii)] $\|[P, Q]^R\|_F= \|[P^R, Q^R]\|_F$.
		\item[\rm (iii)] $\|[P^R, Q^R]\|_F= 2\|[P_c^R, Q_c^R]\|_F$.
	\end{enumerate}
\end{lemma}
\begin{lemma}\label{lem22}
	Let $P= R_1+ R_2\j \in \QR^{m \times n}$ and $Q= S_1+ S_2\j \in \QR^{m \times d}$. Then, the following results hold.
	\begin{enumerate}
		\item[\rm (i)] $\|[P, Q]\|_F= \frac{1}{\sqrt{2}}\|[P, Q]^C\|_F$.
		\item[\rm (ii)] $\|[P, Q]^C\|_F= \|[P^C, Q^C]\|_F$.
		\item[\rm (iii)] $\|[P^C, Q^C]\|_F= \sqrt{2}\|[P_c^C, Q_c^C]\|_F$.
	\end{enumerate}
\end{lemma}

\section{Computation of Real Solutions to the RBTLSE Problem}\label{sec3}
This section is dedicated to exploring an algebraic framework for computing a real solution to the RBTLSE problem by connecting it to its real counterpart, the real TLSE problem. We begin by formulating the RBTLSE problem mathematically. Subsequently, we explore its equivalence with the multidimensional real TLSE problem, laying the groundwork for deriving an efficient solution strategy. 
Suppose 
\begin{eqnarray}
	A&=& A_0+ A_1\i+ A_2\j+ A_3\k \in  \QR^{m \times n}, ~ B= B_0+ B_1\i+ B_2\j+ B_3\k \in \QR^{m \times d}, \label{eq3.1} \\ 
	C&=& C_0+ C_1\i+ C_2\j+ C_3\k \in  \QR^{p \times n}, ~ D= D_0+ D_1\i+ D_2\j+ D_3\k \in \QR^{p \times d}. \label{eq3.2}
\end{eqnarray}
The RBTLSE problem is defined as follows:
\begin{equation}\label{eq3.3}
	\min_{X,\bar{E},\bar{F}}\norm{[\bar{E}, \bar{F}]}_F ~ \textrm{subject to} ~  (A+\bar{E})X = B+\bar{F}, ~ CX=D. 
\end{equation} 
Any real matrix $X$ satisfying the perturbed system in \eqref{eq3.3} with minimal perturbations is called a real RBTLSE solution.

To facilitate numerical computation using real-valued operations, we associate the RBTLSE problem with the following equivalent real TLSE problem:
\begin{equation}\label{eq3.4}
	\min_{X,\widetilde{E}, \widetilde{F}} \norm{[\widetilde{E}, \widetilde{F}]}_F ~ \textrm{subject to} ~ (A^R_c+\widetilde{E})X=B^R_c+ \widetilde{F}, ~ C^R_cX=D^R_c.
\end{equation}
Any real matrix $X$ satisfying the perturbed system in \eqref{eq3.4} with minimal perturbations is called a real TLSE solution.

Throughout this section, we assume that $m \geq n+d$, and that the matrix $C^R_c$ has full row rank. The following theorem establishes a fundamental equivalence between the RBTLSE and real TLSE problems.
\begin{theorem}\label{thm3.1}
	A matrix $X \in \R^{n \times d}$ is a real RBTLSE solution of \eqref{eq3.3} if and only if it is a real TLSE solution of \eqref{eq3.4}. Furthermore, let $X$ be a real TLSE solution with corresponding minimizing perturbations $\widetilde{E}$ and $\widetilde{F}$ in equation \eqref{eq3.4} given by
	\begin{equation*}
		\widetilde{E} = \begin{bmatrix} E_0^T & E_1^T & E_2^T & E_3^T \end{bmatrix}^T \in \R^{4m \times n}, ~
		\widetilde{F} = \begin{bmatrix} F_0^T & F_1^T & F_2^T & F_3^T \end{bmatrix}^T \in \R^{4m \times d},
	\end{equation*}
where $E_t \in \R^{m \times n}$ and $F_t \in \R^{m \times d}$ for $t=0,1,2,3$.

Then the corresponding minimizing perturbations $\bar{E}$ and $\bar{F}$ in equation \eqref{eq3.3} are given by
\begin{equation*}
	\bar{E} = E_0 + E_1\i + E_2\j + E_3\k \in \QR^{m \times n}, ~
	\bar{F} = F_0 + F_1\i + F_2\j + F_3\k \in \QR^{m \times d}.
\end{equation*}
\end{theorem}
\begin{proof}
	Let $X \in \R^{n \times d}$ be a real TLSE solution. Then, there exist perturbation matrices $\widetilde{E}$ and $\widetilde{F}$ satisfying
	\begin{equation*}
		\norm{[\widetilde{E}, \widetilde{F}]}_F= \textrm{min}, ~	(A^R_c+ \widetilde{E})X= B^R_c+ \widetilde{F}, ~ C^R_cX=D^R_c.
	\end{equation*}
	Using these equations, we get
	\begin{align}\label{eq3.5}
		\left[(A^R_c+ \widetilde{E}), \mathcal{K}_m (A^R_c+ \widetilde{E}), \mathcal{L}_m (A^R_c+ \widetilde{E}), \mathcal{M}_m (A^R_c+ \widetilde{E})\right]
		\begin{bmatrix}
			X & 0 & 0 & 0 \\
			0 & X & 0 & 0 \\
			0 & 0 & X & 0 \\
			0 & 0 & 0 & X
		\end{bmatrix} \nonumber &\\= \left[(B^R_c+ \widetilde{F}), \mathcal{K}_m (B^R_c+ \widetilde{F}), \mathcal{L}_m (B^R_c+ \widetilde{F}), \mathcal{M}_m (B^R_c+ \widetilde{F}) \right],
	\end{align}
and
	\begin{equation}\label{eq3.51}
		\left[C^R_c, \mathcal{K}_p C^R_c, \mathcal{L}_p C^R_c, \mathcal{M}_p C^R_c\right]	\begin{bmatrix}
			X & 0 & 0 & 0 \\
			0 & X & 0 & 0 \\
			0 & 0 & X & 0 \\
			0 & 0 & 0 & X
		\end{bmatrix} = \left[D^R_c, \mathcal{K}_p D^R_c, \mathcal{L}_p D^R_c, \mathcal{M}_p D^R_c\right].
	\end{equation}
Next, express the perturbed matrices as
	\begin{equation}\label{eq3.6}
		A^R_c + \widetilde{E}= 
		\begin{bmatrix}
			A_{0}+ E_{0} \\
			A_{1}+ E_{1} \\
			A_{2}+ E_{2} \\
			A_{3}+ E_{3} 
		\end{bmatrix}, ~ B^R_c+\widetilde{F}= 
		\begin{bmatrix}
			B_0+ F_0 \\
			B_1+ F_1  \\
			B_2+ F_2 \\
			B_3+ F_3
		\end{bmatrix}.
	\end{equation}
	Now define the RB matrices:
	\begin{eqnarray*}
		\bar{A}&:=& (A_{0}+E_{0})+ (A_{1}+E_{1})\i+ (A_{2}+E_{2})\j+ (A_{3}+E_{3})\k, \\
		\bar{B}&:=& (B_0+ F_0)+ (B_1+ F_1)\i+ (B_2+ F_2)\j+ (B_3+ F_3)\k, \\
		\bar{E}&:=& E_{0}+ E_{1}\i+ E_{2}\j+ E_{3}\k, ~ \bar{F}:= F_0+ F_1\i+ F_2\j+ F_3\k.
	\end{eqnarray*}
	Using equations \eqref{eq2.2} and \eqref{eq3.6}, we obtain
	\begin{equation}\label{eq3.61}
		\bar{A}_c^R= A^R_c+ \widetilde{E}, ~	\bar{B}_c^R= B^R_c+\widetilde{F}, ~ \bar{E}_c^R=  \widetilde{E}, ~ \textrm{and} ~ \bar{F}_c^R= \widetilde{F}.
	\end{equation}
Furthermore, using equations \eqref{eq2.3} and \eqref{eq3.61}, we can express $\bar{A}^R$, $\bar{B}^R$, and $X^R$ as
	\begin{eqnarray*}
		\bar{A}^R&=& \left[(A^R_c+ \widetilde{E}), \mathcal{K}_m(A^R_c+ \widetilde{E}), \mathcal{L}_m(A^R_c+ \widetilde{E}), \mathcal{M}_m(A^R_c+ \widetilde{E})\right], \\
		\bar{B}^R&=& \left[(B^R_c+\widetilde{F}), \mathcal{K}_m(B^R_c+\widetilde{F}), \mathcal{L}_m(B^R_c+\widetilde{F}), \mathcal{M}_m(B^R_c+\widetilde{F})\right],\\
		X^R&=& 
		\begin{bmatrix}
			X & 0 & 0 & 0 \\
			0 & X & 0 & 0 \\
			0 & 0 & X & 0 \\
			0 & 0 & 0 & X
		\end{bmatrix}.
	\end{eqnarray*}
	Thus, equation \eqref{eq3.5} can be rewritten as
	\begin{eqnarray}
		\bar{A}^R X^R&=& \bar{B}^R, \label{eq3.7}\\
		(\bar{A}X)^R&=& \bar{B}^R, \nonumber\\
		\bar{A}X&=& \bar{B}. \label{eq3.8}
	\end{eqnarray}
	The matrices $\bar{A}$ and $\bar{B}$ can be represented using $A$, $B$, $\bar{E}$, and $\bar{F}$ as follows
	\begin{eqnarray}
		\bar{A}&=& (A_{0}+ E_{0})+ (A_{1}+ E_{1})\i+ (A_{2}+ E_{2})\j+ (A_{3}+ E_{3})\k \nonumber \\
		&=& (A_0+ A_1\i+ A_2\j+ A_3\k)+ (E_0 +E_1\i+ E_2\j+ E_3\k) = A+ \bar{E}, \label{eq3.9}
	\end{eqnarray}
	and
	\begin{eqnarray}
		\bar{B}&=& (B_0+ F_0)+ (B_1+ F_1)\i+ (B_2+ F_2)\j+ (B_3+ F_3)\k \nonumber\\
		&=& (B_0+ B_1\i+ B_2\j+ B_3\k)+ (F_0 + F_1\i+ F_2\j+ F_3\k) = B+ \bar{F}. \label{eq3.10}
	\end{eqnarray}
	Using \eqref{eq3.9} and \eqref{eq3.10}, equation \eqref{eq3.8} is equivalent to
	\begin{equation} \label{eq3.11}
		(A+\bar{E})X= B+ \bar{F}.
	\end{equation}
	Also, using equation \eqref{eq2.3}, we have
	\begin{equation*}
		C^R = \left[C^R_c, \mathcal{K}_p C^R_c, \mathcal{L}_p C^R_c, \mathcal{M}_p C^R_c\right], ~ D^R=\left[D^R_c, \mathcal{K}_p D^R_c, \mathcal{L}_p D^R_c, \mathcal{M}_p D^R_c\right].
	\end{equation*}
Thus, equation \eqref{eq3.51} is equivalent to
	\begin{eqnarray}
		C^R X^R&=& D^R, \nonumber\\
		(CX)^R&=& D^R, \nonumber\\
		CX&=& D. \label{eq33.13}
	\end{eqnarray}
	Finally, using Lemma~\ref{lem2}, we relate the norms 
	\begin{equation}\label{eq3.12}
		\norm{[\bar{E}, \bar{F}]}_F= \frac{1}{2} \norm{[\bar{E}, \bar{F}]^R}_F= \frac{1}{2} \norm{[\bar{E}^R, \bar{F}^R]}_F =\norm{[\bar{E}^R_c, \bar{F}^R_c]}_F=  \norm{[\widetilde{E}, \widetilde{F}]}_F= \min.  
	\end{equation}
	Combining \eqref{eq3.11}, \eqref{eq33.13}, and \eqref{eq3.12}, we conclude that there exist matrices $\bar{E} \in \QR^{m \times n}$ and $\bar{F} \in \QR^{m \times d}$ such that
	$X  \in  \R^{n \times d}$ is a real RBTLSE solution, and vice versa.
\end{proof}
To proceed with obtaining explicit expressions for $X$, we begin by introducing the following matrices:
\begin{equation}\label{eq3.122}
 \hat{P}=\left[A^R_c, B^R_c\right], ~ \hat{S}=\left[C^R_c, D^R_c\right], ~ \hat{E} = \left[\widetilde{E}, ~ \widetilde{F}\right], ~ \text{and} ~ \bar{X}=\left[X^T, -I_d\right]^T.
 \end{equation}
 Now, we present the following theorem that provides an explicit expression for the real RBTLSE solution.
\begin{theorem}\label{thm3.2}
	Consider the RBTLSE problem \eqref{eq3.3}, with the notations introduced in \eqref{eq3.122}. Consider the QR factorization of $\hat{S}^T$ given by
	\begin{equation}\label{eqth1}
		\hat{S}^T = \hat{Q}\begin{bmatrix}
			\hat{R} \\
			0             
		\end{bmatrix},
	\end{equation}
	where $\hat{Q} = \left[\hat{Q}_1, \hat{Q}_2\right]$ with  $\hat{Q}_1 \in \R^{(n+d) \times 4p}$ and $\hat{Q}_2 \in \R^{(n+d) \times (n+d-4p)}$. Let the thin SVD of $\hat{P}\hat{Q}_2$ be given by 
	\begin{equation}\label{eq23}
		\hat{P}\hat{Q}_2= \hat{U} \hat{\Sigma} \hat{V}^T,
	\end{equation} where $\hat{U} \in \R^{4m \times (n-4p+d)}$ and $\hat{V} \in \R^{(n-4p+d) \times (n-4p+d)}$ are real orthonormal matrices, and $\hat{\Sigma}= \diag(\hat{\sigma}_1, \hat{\sigma}_2, \ldots,  \hat{\sigma}_{n-4p},  \hat{\sigma}_{n-4p+1}, \ldots, \hat{\sigma}_{n-4p+d})$ is a diagonal matrix with singular values satisfying 
	\begin{equation}\label{eq23.1}
		\hat{\sigma}_1 \geq \hat{\sigma}_2 \geq \ldots \geq  \hat{\sigma}_{n-4p} >  \hat{\sigma}_{n-4p+1} \geq \ldots \geq \hat{\sigma}_{n-4p+d} > 0.
	\end{equation}  The matrices $\hat{U}$, $\hat{\Sigma}$, and $\hat{V}$ are expressed in block form as
	\vspace{-0.5cm}
	\begin{equation}\label{eq24}
		\hat{U}= \begin{blockarray}{c@{}cc@{\hspace{4pt}}cl}
			&  \mLabel{n-4p} & \mLabel{d} & &\\
			\begin{block}{[c@{\hspace{5pt}}cc@{\hspace{5pt}}c]l}
				&\hat{U}_1& \hat{U}_2& & \mLabel{4m} \\
			\end{block}
		\end{blockarray}, ~
		\hat{\Sigma}= \begin{blockarray}{c@{}cc@{\hspace{4pt}}cl}
			&  \mLabel{n-4p} & \mLabel{d} & &\\
			\begin{block}{[c@{\hspace{5pt}}cc@{\hspace{5pt}}c]l}
				&\hat{\Sigma}_1& 0& & \mLabel{n-4p} \\
				&0 & \hat{\Sigma}_2 & & \mLabel{d} \\
			\end{block}
		\end{blockarray}, ~
		\hat{V}= \begin{blockarray}{c@{}cc@{\hspace{4pt}}cl}
			&  \mLabel{n-4p} & \mLabel{d} & &\\
			\begin{block}{[c@{\hspace{5pt}}cc@{\hspace{5pt}}c]l}
				&\hat{V}_1& \hat{V}_2& & \mLabel{n-4p+d} \\
			\end{block}
		\end{blockarray},
	\end{equation}
	where $\hat{\Sigma}_1= \diag(\hat{\sigma}_1, \hat{\sigma}_2, \ldots,  \hat{\sigma}_{n-4p})$ and $\hat{\Sigma}_2= \diag(\hat{\sigma}_{n-4p+1}, \hat{\sigma}_{n-4p+2}, \ldots,  \hat{\sigma}_{n-4p+d})$. Define
	\begin{equation}\label{eq24.1}
		\check{V}=\hat{Q}_2 \hat{V} = \begin{bmatrix}
			\check{V}_1  & \check{V}_2
		\end{bmatrix}= \begin{blockarray}{c@{}cc@{\hspace{4pt}}cl}
			&  \mLabel{n-4p} & \mLabel{d} & &\\
			\begin{block}{[c@{\hspace{5pt}}cc@{\hspace{5pt}}c]l}
				& \check{V}_{11} &  \check{V}_{12} & & \mLabel{n} \\
				& \check{V}_{21}  &  \check{V}_{22}  & & \mLabel{d} \\
			\end{block}
		\end{blockarray}.
	\end{equation}
 If $\hat{\sigma}_{n-4p} >  \hat{\sigma}_{n-4p+1}$ and the matrix $\check{V}_{22}$ is invertible, then a unique real RBTLSE solution $X\in \R^{n \times d}$ exists. The corresponding expression for $X$ in this scenario is 
	\begin{equation}\label{eq26}
		X=-\check{V}_{12} \check{V}_{22}^{-1}.
	\end{equation}
The corresponding minimizing perturbations $\bar{E} = E_0 + E_1 \i + E_2 \j + E_3 \k$ and $\bar{F} = F_0 + F_1 \i + F_2 \j + F_3 \k$ are given by
	\begin{equation*}
		[E_{0}^T, E_{1}^T, E_{2}^T, E_{3}^T]^T= -\hat{U}_2 \hat{\Sigma}_2 \check{V}_{12}^T, ~  [F_{0}^T, F_{1}^T, F_{2}^T, F_{3}^T]^T= -\hat{U}_2 \hat{\Sigma}_2 \check{V}_{22}^T.
	\end{equation*}
\end{theorem}
\begin{proof}
	To find real RBTLSE solution, we need to find the real TLSE solution. We can reformulate equation \eqref{eq3.4} as
	\begin{eqnarray}
		\min_{ \bar{X}, \widetilde{E},\widetilde{F}} \norm{[\widetilde{E}, \widetilde{F}]}_F ~ &\textrm{subject to}& ~ \left[A^R_c+\widetilde{E}, B^R_c+\widetilde{F}\right]\bar{X} =  0, ~ \left[C^R_c, D^R_c\right]\bar{X}=0, \nonumber \\
		\min_{ \bar{X}, \hat{E}} \norm{\hat{E}}_F ~ &\textrm{subject to}& ~ \left(\hat{P}+\hat{E} \right)\bar{X} =  0, ~ \hat{S}\bar{X}=0. \label{eqpfr2.1} 
	\end{eqnarray}
From the constraint $\hat{S}\bar{X} = 0$, we can deduce that the column range $\mathcal{R}(\bar{X})$ must lie in the null space of $\hat{S}$. Since $\mathcal{N}(\hat{S}) = \mathcal{R}(\hat{S}^T)^{\perp}$, an orthonormal basis for $\mathcal{N}(\hat{S})$ can be obtained via QR factorization of $\hat{S}^T$.

	The QR factorization of $\hat{S}^T$ from \eqref{eqth1} implies that $\hat{Q}_2$ spans $\mathcal{N}(\hat{S})$. Therefore, there exists a matrix $Z \in \mathbb{R}^{(n+d-4p) \times d}$ such that $\bar{X} = \hat{Q}_2 Z$. Substituting this into equation \eqref{eqpfr2.1}, we obtain
		\begin{eqnarray}
		\min_{ \bar{X}, \hat{E}} \norm{\hat{E}}_F ~ &\textrm{subject to}& ~ \left(\hat{P}+\hat{E}\right)\hat{Q}_2Z = 0, \nonumber \\
		\min_{ \bar{X}, \hat{E}} \norm{\left[\hat{E}\hat{Q}_1, \hat{E}\hat{Q}_2\right]}_F ~ &\textrm{subject to}& ~ \left(\hat{P}\hat{Q}_2+\hat{E}\hat{Q}_2\right)Z = 0. \label{eqth21}
	\end{eqnarray}
Since the constraint only involves $\hat{E}\hat{Q}_2$, we can choose an optimal matrix $\hat{E}_{*}$ such that $\hat{E}_{*}\hat{Q}_1 = 0$ while ensuring that $\hat{P}\hat{Q}_2 + \hat{E}_{*}\hat{Q}_2$ has a null space of dimension at least $d$. The condition $\hat{E}_{*}\hat{Q}_1 = 0$ indicates that there exists a matrix $Y \in \mathbb{R}^{4m \times (n+d-4p)}$ such that $\hat{E}_{*} = Y\hat{Q}_{2}^T$. Thus, equation \eqref{eqth21} becomes
	\begin{equation}\label{eqth31}
		\min_{Y}\norm{Y}_F ~ \textrm{subject to} ~ \left(\hat{P}\hat{Q}_2+Y\right)Z=0.
	\end{equation}
To ensure that the null space of $\hat{P}\hat{Q}_2+Y$ has dimension at least $d$, by the rank-nullity theorem, we need to reduce the rank of $\hat{P}\hat{Q}_2+Y$ to $n-4p$. Hence, equation \eqref{eqth31} can be rewritten as
	\begin{equation}\label{eqth41}
		\min_{\rank(\hat{P}\hat{Q}_2 +Y) \leq n-4p}\norm{Y}_F ~ \textrm{subject to} ~ \left(\hat{P}\hat{Q}_2+Y\right)Z=0.
	\end{equation}
Let $Y_{*}$ be the matrix such that $\hat{P}\hat{Q}_2+Y_{*}$ is the best rank $n-4p$ approximation of $\hat{P}\hat{Q}_2$. By the Eckart-Young-Mirsky theorem and the SVD of $\hat{P}\hat{Q}_2$ given in \eqref{eq23}, we have
\begin{equation*}
	\hat{P}\hat{Q}_2+Y_{*} = \hat{U}_1 \hat{\Sigma}_1\hat{V}_1^T.
\end{equation*}
When $\hat{\sigma}_{n-4p} > \hat{\sigma}_{n-4p+1}$, this provides the unique rank $n-4p$ approximation of $\hat{P}\hat{Q}_2$. Equation \eqref{eqth41} thus simplifies to
\begin{equation*}
		\left(\hat{U}_1 \hat{\Sigma}_1\hat{V}_1^T\right)Z=0.
	\end{equation*}	
Since $\hat{V}_1^T\hat{V}_2=0$ (orthogonality of singular vectors), we can set $Z=\hat{V}_2$. This gives us
$$\bar{X} = \hat{Q}_2Z=\hat{Q}_2\hat{V}_2=\check{V}_2=\begin{bmatrix}
		\check{V}_{12} \\
		\check{V}_{22}
	\end{bmatrix}.$$
	If $\check{V}_{22}$ is invertible, then we obtain
	\begin{equation*}
		X=-\check{V}_{12}\check{V}_{22}^{-1}.
	\end{equation*}
	For the matrix $Y_{*}$, we have
	\begin{eqnarray*}
		Y_{*} &=& \begin{bmatrix}
			\hat{U}_1 & \hat{U}_2
		\end{bmatrix}\begin{bmatrix}
			\hat{\Sigma}_1 & 0\\
			0 & 0
		\end{bmatrix}\begin{bmatrix}
			\hat{V}_1 & \hat{V}_2
		\end{bmatrix}^T - \begin{bmatrix}
			\hat{U}_1 & \hat{U}_2
		\end{bmatrix}\begin{bmatrix}
			\hat{\Sigma}_1 & 0\\
			0 & \hat{\Sigma}_2
		\end{bmatrix}\begin{bmatrix}
			\hat{V}_1 & \hat{V}_2
		\end{bmatrix}^T \\
		& =& -\hat{U}_2 \hat{\Sigma}_2 \hat{V}_2^T.
	\end{eqnarray*} 
Therefore, the minimizing perturbation matrices are
	\begin{equation*}
		\hat{E}_{*} =[\widetilde{E}, \widetilde{F}]=Y_*\hat{Q}_2^T=-\hat{U}_2 \hat{\Sigma}_2 \hat{V}_2^T \hat{Q}_2^T = -\hat{U}_2 \hat{\Sigma}_2(\hat{Q}_2 \hat{V}_2)^T = -\hat{U}_2 \hat{\Sigma}_2 \check{V}_2^T = [-\hat{U}_2 \hat{\Sigma}_2 \check{V}_{12}^T, -\hat{U}_2 \hat{\Sigma}_2 \check{V}_{22}^T].
	\end{equation*}
	Hence, we obtain
	\begin{equation*}
		\widetilde{E} = \begin{bmatrix} E_0^T & E_1^T & E_2^T & E_3^T \end{bmatrix}^T= -\hat{U}_2 \hat{\Sigma}_2 \check{V}_{12}^T, ~  \widetilde{F} = \begin{bmatrix} F_0^T & F_1^T & F_2^T & F_3^T \end{bmatrix}^T= -\hat{U}_2 \hat{\Sigma}_2 \check{V}_{22}^T.
	\end{equation*}
\end{proof}

\section{Computation of Complex Solutions to the RBTLSE Problem}\label{sec4}
This section focuses on developing an algebraic framework for computing a complex solution to the RBTLSE problem by establishing its connection with the corresponding complex TLSE problem. We begin by presenting the mathematical formulation of the RBTLSE problem. Then, we demonstrate its equivalence to a multidimensional complex TLSE problem, which serves as the foundation for deriving an efficient solution method. Suppose 
\begin{eqnarray}
	A&=& M_1+M_2\j \in  \QR^{m \times n}, ~ B= N_1+N_2\j \in \QR^{m \times d}, \label{eq4.1} \\ 
	C&=& R_1+R_2\j \in  \QR^{p \times n}, ~ D= S_1+S_2\j \in \QR^{p \times d}. \label{eq4.2}
\end{eqnarray}
The RBTLSE problem is defined as follows:
\begin{equation}\label{eq4.3}
	\min_{X,\bar{G},\bar{H}}\norm{[\bar{G}, \bar{H}]}_F ~ \textrm{subject to} ~ (A+\bar{G})X = B+\bar{H}, ~ CX=D. 
\end{equation} 
Any complex matrix $X$ satisfying the perturbed system in \eqref{eq4.3} with minimal perturbations is called a complex RBTLSE solution.

To facilitate computation within the complex domain, we introduce the corresponding complex TLSE problem:
\begin{equation}\label{eq4.4}
	\min_{X,\widetilde{G}, \widetilde{H}} \norm{[\widetilde{G}, \widetilde{H}]}_F ~ \textrm{subject to} ~ (A^C_c+\widetilde{G})X=B^C_c+ \widetilde{H}, ~ C^C_cX=D^C_c.
\end{equation}
Any complex matrix $X$ satisfying the perturbed system in \eqref{eq4.4} with minimal perturbations is called a complex TLSE solution.

Throughout this section, we assume that $m \geq n + d$, and that the matrix $C^C_c$ has full row rank. The following theorem establishes a fundamental equivalence between the RBTLSE problem and its complex TLSE counterpart.
\begin{theorem}\label{thm4.1}
	A matrix $X \in \C^{n \times d}$ is a complex RBTLSE solution of \eqref{eq4.3} if and only if it is a complex TLSE solution of \eqref{eq4.4}. Furthermore, let $X$ be a complex TLSE solution with corresponding minimizing perturbations $\widetilde{G}$ and $\widetilde{H}$ in equation \eqref{eq4.4} given by
	\begin{equation*}
		\widetilde{G} = \begin{bmatrix} G_1^T & G_2^T \end{bmatrix}^T \in \C^{2m \times n}, ~
		\widetilde{H} = \begin{bmatrix} H_1^T & H_2^T \end{bmatrix}^T \in \C^{2m \times d},
	\end{equation*}
	where $G_1, G_2 \in \C^{m \times n}$ and $H_1, H_2 \in \C^{m \times d}$.
	
	 Then the corresponding minimizing perturbations $\bar{G}$ and $\bar{H}$ in equation \eqref{eq4.3} are given by
	\begin{equation*}
		\bar{G} = G_1 + G_2\j \in \QR^{m \times n}, ~
		\bar{H} = H_1 + H_2\j \in \QR^{m \times d}.
	\end{equation*}
	\end{theorem}
\begin{proof}
Let $X \in \C^{n \times d}$ be a complex TLSE solution. Then, there exist perturbation matrices $\widetilde{G} \in \C^{2m \times n}$ and $\widetilde{H} \in \C^{2m \times d}$ satisfying
	\begin{equation*}
		\norm{[\widetilde{G}, \widetilde{H}]}_F= \textrm{min}, ~	(A^C_c+ \widetilde{G})X= B^C_c+ \widetilde{H}, ~ C^C_cX=D^C_c.
	\end{equation*}
Using these equations, we get
	\begin{equation}\label{eq4.5}
		\left[(A^C_c+ \widetilde{G}), \mathcal{N}_m (A^C_c+ \widetilde{G})\right]
		\begin{bmatrix}
			X & 0  \\
			0 & X 
		\end{bmatrix} = \left[(B^C_c+ \widetilde{H}), \mathcal{N}_m (B^C_c+ \widetilde{H})\right],
	\end{equation}
and
	\begin{equation}\label{eq4.51}
		\left[C^C_c, \mathcal{N}_p C^C_c\right]	\begin{bmatrix}
			X & 0  \\
			0 & X 
		\end{bmatrix} = \left[D^C_c, \mathcal{N}_p D^C_c\right].
	\end{equation}
	Next, express the perturbed matrices as
	\begin{equation}\label{eq4.6}
		A^C_c + \widetilde{G}= 
		\begin{bmatrix}
			M_{1}+ G_{1} \\
			M_{2}+ G_{2} 
		\end{bmatrix}, ~ B^C_c+\widetilde{H}= 
		\begin{bmatrix}
			N_1+ H_1  \\
			N_2+ H_2 
		\end{bmatrix}.
	\end{equation}
	Now define the RB matrices:
	\begin{eqnarray*}
		\bar{A}&:=& (M_{1}+G_{1})+ (M_{2}+G_{2})\j, \\
		\bar{B}&:=& (N_1+ H_1)+ (N_2+ H_2)\j, \\
		\bar{G}&:=& G_{1}+ G_{2}\j, ~ \bar{H}:= H_1+ H_2\j.
	\end{eqnarray*}
	Using equations \eqref{eq2.2} and \eqref{eq4.6}, we obatin
	\begin{equation*}
		\bar{A}_c^C= A^C_c+ \widetilde{G}, ~	\bar{B}_c^C= B^C_c+\widetilde{H}, ~ \bar{G}_c^C=  \widetilde{G}, ~ \textrm{and} ~ \bar{H}_c^C= \widetilde{H}. 
	\end{equation*}
	Furthermore, using equations \eqref{eq2.4} and \eqref{eq4.6}, we can express $\bar{A}^C$, $\bar{B}^C$, and $X^C$ as
	\begin{eqnarray*}
		\bar{A}^C&=& \left[(A^C_c+ \widetilde{G}), \mathcal{N}_m(A^C_c+ \widetilde{G})\right], \\
		\bar{B}^C&=& \left[(B^C_c+\widetilde{H}), \mathcal{N}_m(B^C_c+\widetilde{H})\right],\\
		X^C&=& 
		\begin{bmatrix}
			X & 0  \\
			0 & X 
		\end{bmatrix}.
	\end{eqnarray*}
	Thus, equation \eqref{eq4.5} can be rewritten as
	\begin{eqnarray}
		\bar{A}^C X^C&=& \bar{B}^C, \label{eq4.7}\\
		(\bar{A}X)^C&=& \bar{B}^C, \nonumber\\
		\bar{A}X&=& \bar{B}. \label{eq4.8}
	\end{eqnarray}
	Now, we can express $\bar{A}$ and $\bar{B}$ in terms of $A$, $B$, $\bar{G}$, and $\bar{H}$ as
	\begin{eqnarray}
		\bar{A}&=& (M_{1}+ G_{1})+ (M_{2}+ G_{2})\j \nonumber \\
		&=& (M_1+ M_2\j)+ (G_1+ G_2\j) = A+ \bar{G}, \label{eq4.9}
	\end{eqnarray}
	and
	\begin{eqnarray}
		\bar{B}&=& (N_1+ H_1)+ (N_2+ H_2)\j \nonumber\\
		&=& (N_1+ N_2\j)+ (H_1+ H_2\j) = B+ \bar{H}. \label{eq4.10}
	\end{eqnarray}
	Using \eqref{eq4.9} and \eqref{eq4.10}, equation \eqref{eq4.8} is equivalent to
	\begin{equation} \label{eq4.11}
		(A+\bar{G})X= B+ \bar{H}.
	\end{equation}
Also, using equation \eqref{eq2.4}, we have
	\begin{equation}\label{eq4.12}
		C^C = \left[C^C_c, \mathcal{N}_p C^C_c\right], ~ D^C=\left[D^C_c, \mathcal{N}_p D^C_c\right].
	\end{equation}
	Thus, equation \eqref{eq4.51} is equivalent to
	\begin{eqnarray}
		C^C X^C&=& D^C, \nonumber\\
		(CX)^C&=& D^C, \nonumber\\
		CX&=& D. \label{eq44.13}
	\end{eqnarray}
	Finally, using Lemma \ref{lem22}, we relate the norms
	\begin{equation}\label{eq4.121}
		\norm{[\bar{G}, \bar{H}]}_F= \frac{1}{\sqrt{2}} \norm{[\bar{G}, \bar{H}]^C}_F= \frac{1}{\sqrt{2}} \norm{[\bar{G}^C, \bar{H}^C]}_F= \norm{[\bar{G}^C_c, \bar{H}^C_c]}_F= \norm{[\widetilde{G}, \widetilde{H}]}_F= \min.  
	\end{equation}
	Combining \eqref{eq4.11}, \eqref{eq44.13}, and \eqref{eq4.121}, we conclude that there exist matrices $\bar{G} \in \QR^{m \times n}$ and $\bar{H} \in \QR^{m \times d}$ such that
	$X  \in  \C^{n \times d}$ is a complex RBTLSE solution, and vice versa.
\end{proof}
To proceed with obtaining explicit expressions for $X$, we begin by introducing the following matrices:
\begin{equation}\label{eqmat}
\grave{P} = \left[A^C_c, B^C_c\right], ~  \grave{S} = \left[C^C_c, D^C_c\right], ~ \grave{G} = \left[\widetilde{G}, \widetilde{H}\right], ~ \text{and} ~ \bar{X}=\left[X^T, -I_d\right]^T.
\end{equation}
The following theorem establishes an explicit formula for the complex RBTLSE solution.
\begin{theorem}\label{thm4.2}
	Consider the RBTLSE problem \eqref{eq4.3}, with the notations introduced in \eqref{eqmat}. Consider the QR factorization of \(\grave{S}^H\) given by
	\begin{equation}\label{4.15}
		\grave{S}^H = \grave{Q}\begin{bmatrix}
			\grave{R} \\
			0             
		\end{bmatrix},
	\end{equation}
	where $\grave{Q} = \left[\grave{Q}_1, \grave{Q}_2\right]$ with $\grave{Q}_1 \in \C^{(n+d) \times 2p}$ and $\grave{Q}_2 \in \C^{(n+d) \times (n+d-2p)}$.
	Let the thin SVD of $\grave{P}\grave{Q}_2$ be given by 
	\begin{equation}\label{eq4.16}
		\grave{P}\grave{Q}_2= \grave{U} \grave{\Sigma} \grave{V}^H,
	\end{equation} where $\grave{U}$ and $\grave{V}$ are unitary matrices, and $\grave{\Sigma}= \diag(\grave{\sigma}_1, \grave{\sigma}_2, \ldots,  \grave{\sigma}_{n-2p},  \grave{\sigma}_{n-2p+1}, \ldots, \grave{\sigma}_{n-2p+d})$ is a diagonal matrix with singular values satisfying
	\begin{equation}\label{eq4.17}
		\grave{\sigma}_1 \geq \grave{\sigma}_2 \geq \ldots \geq  \grave{\sigma}_{n-2p} >  \grave{\sigma}_{n-2p+1} \geq \ldots \geq \grave{\sigma}_{n-2p+d} > 0.
	\end{equation} 
The matrices $\grave{U}$, $\grave{\Sigma}$, and $\grave{V}$ are expressed in block form as
	\vspace{-0.5cm}
	\begin{equation}\label{eq4.18}
		\grave{U}= \begin{blockarray}{c@{}cc@{\hspace{4pt}}cl}
			&  \mLabel{n-2p} & \mLabel{d} & &\\
			\begin{block}{[c@{\hspace{5pt}}cc@{\hspace{5pt}}c]l}
				&\grave{U}_1& \grave{U}_2& & \mLabel{2m} \\
			\end{block}
		\end{blockarray},
		\grave{\Sigma}= \begin{blockarray}{c@{}cc@{\hspace{4pt}}cl}
			&  \mLabel{n-2p} & \mLabel{d} & &\\
			\begin{block}{[c@{\hspace{5pt}}cc@{\hspace{5pt}}c]l}
				&\grave{\Sigma}_1& 0& & \mLabel{n-2p} \\
				&0 & \grave{\Sigma}_2 & & \mLabel{d} \\
			\end{block}
		\end{blockarray},
		\grave{V}= \begin{blockarray}{c@{}cc@{\hspace{4pt}}cl}
			&  \mLabel{n-2p} & \mLabel{d} & &\\
			\begin{block}{[c@{\hspace{5pt}}cc@{\hspace{5pt}}c]l}
				&\grave{V}_1& \grave{V}_2& & \mLabel{n-2p+d} \\
			\end{block}
		\end{blockarray},
	\end{equation}
	where $\grave{\Sigma}_1= \diag(\grave{\sigma}_1, \grave{\sigma}_2, \ldots,  \grave{\sigma}_{n-2p})$ and $\grave{\Sigma}_2= \diag(\grave{\sigma}_{n-2p+1}, \grave{\sigma}_{n-2p+2}, \ldots,  \grave{\sigma}_{n-2p+d})$. Define
	\begin{equation}\label{eq4.19}
		\acute{V}=\grave{Q}_2 \grave{V} = \begin{bmatrix}
			\acute{V}_1  & \acute{V}_2
		\end{bmatrix}= \begin{blockarray}{c@{}cc@{\hspace{4pt}}cl}
			&  \mLabel{n-2p} & \mLabel{d} & &\\
			\begin{block}{[c@{\hspace{5pt}}cc@{\hspace{5pt}}c]l}
				& \acute{V}_{11} &  \acute{V}_{12} & & \mLabel{n} \\
				& \acute{V}_{21}  &  \acute{V}_{22}  & & \mLabel{d} \\
			\end{block}
		\end{blockarray}.
	\end{equation}
		If \(\grave{\sigma}_{n-2p} > \grave{\sigma}_{n-2p+1}\) and the matrix \(\acute{V}_{22}\) is invertible, then a unique complex RBTLSE solution $X \in \C^{n \times d}$ exists. The corresponding expression for $X$ in this scenario is
	\begin{equation}\label{comsol}
	X = -\acute{V}_{12} \acute{V}_{22}^{-1}.
	\end{equation}
The corresponding minimizing perturbations \(\bar{G} = G_1 + G_2 \j\) and \(\bar{H} = H_1 + H_2 \j\) are given by
	\[
	\begin{bmatrix} G_1 \\ G_2 \end{bmatrix} = -\grave{U}_2 \grave{\Sigma}_2 \acute{V}_{12}^H, ~
	\begin{bmatrix} H_1 \\ H_2 \end{bmatrix} = -\grave{U}_2 \grave{\Sigma}_2 \acute{V}_{22}^H.
	\]
\end{theorem}
\begin{proof}
	To find complex RBTLSE solution, we need to find the complex TLSE solution. We can reformulate equation \eqref{eq4.4} as
	\begin{eqnarray}
		\min_{ \bar{X}, \widetilde{G},\widetilde{H}} \norm{[\widetilde{G}, \widetilde{H}}_F ~ &\textrm{subject to}& ~ \left[A^C_c+\widetilde{G}, B^C_c+\widetilde{H}\right]\bar{X} =  0, ~ \left[C^C_c, D^C_c\right]\bar{X}=0, \nonumber \\
		\min_{ \bar{X}, \grave{G}} \norm{\grave{G}}_F ~ &\textrm{subject to}& ~ \left(\grave{P}+\grave{G} \right)\bar{X} =  0, ~ \grave{S}\bar{X}=0. \label{eqpf2.1} 
	\end{eqnarray}
	From the constraint $\grave{S} \bar{X} = 0$, we can deduce that the column range $\mathcal{R}(\bar{X})$ must lie in the null space of $\grave{S}$. Since $\mathcal{N}(\grave{S}) = \mathcal{R}(\grave{S}^H)^{\perp}$, an orthonormal basis for $\mathcal{N}(\grave{S})$ can be obtained via QR factorization of $\grave{S}^H$. 
	
	The QR factorization of $\grave{S}^H$ from \eqref{4.15} implies that $\grave{Q}_2$ spans $\mathcal{N}(\grave{S})$. Therefore, there exists a matrix $Z \in \C^{(n+d-2p) \times d}$ such that $\bar{X} = \grave{Q}_2 Z$. Substituting this into equation \eqref{eqpf2.1}, we obtain
	\begin{eqnarray}
		\min_{ \bar{X}, \grave{G}} \norm{\grave{G}}_F ~ &\textrm{subject to}& ~ \left(\grave{P}+\grave{G}\right)\grave{Q}_2Z = 0, \nonumber \\
		\min_{ \bar{X}, \grave{G}} \norm{\left[\grave{G}\grave{Q}_1, \grave{G}\grave{Q}_2\right]}_F ~ &\textrm{subject to}& ~ \left(\grave{P}\grave{Q}_2+\grave{G}\grave{Q}_2\right)Z = 0. \label{eqth2}
	\end{eqnarray}
Since the constraint only involves $\grave{G}\grave{Q}_2$, we can choose an optimal matrix $\grave{G}_{*}$ such that $\grave{G}_{*}\grave{Q}_1 =0$ while ensuring that $\grave{P}\grave{Q}_2 + \grave{G}_{*}\grave{Q}_2$ has a null space of dimension at least $d$. The condition $\grave{G}_{*}\grave{Q}_1=0$ indicates that there exists a matrix $Y \in \mathbb{C}^{2m \times (n+d-2p)}$ such that $\grave{G}_{*} = Y\grave{Q}_{2}^H$. Thus, equation \eqref{eqth2} becomes
	\begin{equation}\label{eqth3}
		\min_{Y}\norm{Y}_F ~ \textrm{subject to} ~ \left(\grave{P}\grave{Q}_2+Y\right)Z=0.
	\end{equation}
To ensure that the null space of $\grave{P}\grave{Q}_2+Y$ has dimension at least $d$, by the rank-nullity theorem, we need to reduce the rank of $\grave{P}\grave{Q}_2+Y$ to $n-2p$. Hence, equation \eqref{eqth3} can be rewritten as
	\begin{equation}\label{eqth4}
		\min_{\rank(\grave{P}\grave{Q}_2 +Y) \leq n-2p}\norm{Y}_F ~ \textrm{subject to} ~ \left(\grave{P}\grave{Q}_2+Y\right)Z=0.
	\end{equation}
Let $Y_{*}$ be the matrix such that $\grave{P}\grave{Q}_2+Y_{*}$ is the best rank $n-2p$ approximation of $\grave{P}\grave{Q}_2$. By the Eckart-Young-Mirsky theorem and the SVD of $\grave{P}\grave{Q}_2$ given in \eqref{eq4.16}, we have
\begin{equation*}
\grave{P}\grave{Q}_2+Y_{*}  = \grave{U}_1 \grave{\Sigma_1}\grave{V}_1^H.
\end{equation*}
When $\hat{\sigma}_{n-2p} > \hat{\sigma}_{n-2p+1}$, this provides the unique rank $n-2p$ approximation of $\grave{P}\grave{Q}_2$. Equation \eqref{eqth4} thus simplifies to
\begin{equation*}
		\left(\grave{U}_1 \grave{\Sigma}_1\grave{V}_1^H\right)Z=0.
\end{equation*}	
Since $\grave{V}_1^H\grave{V}_2=0$ (orthogonality of singular vectors), we can set $Z=\grave{V}_2$. This gives us
	$$\bar{X} = \grave{Q}_2Z=\grave{Q}_2\grave{V}_2=\acute{V}_2=\begin{bmatrix}
		\acute{V}_{12} \\
		\acute{V}_{22}
	\end{bmatrix}.$$
	If $\acute{V}_{22}$ is invertible, then we obtain
	\begin{equation*}
		X=-\acute{V}_{12}\acute{V}_{22}^{-1}.
	\end{equation*}
	For the matrix $Y_{*}$, we have
	\begin{eqnarray*}
		Y_{*} &=& \begin{bmatrix}
			\grave{U}_1 & \grave{U}_2
		\end{bmatrix}\begin{bmatrix}
			\grave{\Sigma}_1 & 0\\
			0 & 0
		\end{bmatrix}\begin{bmatrix}
			\grave{V}_1 & \grave{V}_2
		\end{bmatrix}^H - \begin{bmatrix}
			\grave{U}_1 & \grave{U}_2
		\end{bmatrix}\begin{bmatrix}
			\grave{\Sigma}_1 & 0\\
			0 & \grave{\Sigma}_2
		\end{bmatrix}\begin{bmatrix}
			\grave{V}_1 & \grave{V}_2
		\end{bmatrix}^H \\
		& =& -\grave{U}_2 \grave{\Sigma}_2 \grave{V}_2^H.
	\end{eqnarray*} 
Therefore, the minimizing perturbation matrices are
	\begin{equation*}
		\grave{G}_{*} =[\widetilde{G}, \widetilde{H}]=Y_*\grave{Q}_2^H=-\grave{U}_2 \grave{\Sigma}_2 \grave{V}_2^H \grave{Q}_2^H = -\grave{U}_2 \grave{\Sigma}_2(\grave{Q}_2 \grave{V}_2)^H = -\grave{U}_2 \grave{\Sigma}_2 \acute{V}_2^H = [-\grave{U}_2 \grave{\Sigma}_2 \acute{V}_{12}^H, -\grave{U}_2 \grave{\Sigma}_2 \acute{V}_{22}^H].
	\end{equation*}
	Hence, we obtain
	\begin{equation*}
		\widetilde{G} = \begin{bmatrix}
			G_1 \\ G_2
		\end{bmatrix}= - \grave{U}_2 \grave{\Sigma}_2 \acute{V}_{12}^H, ~   \widetilde{H} = \begin{bmatrix}
			H_1 \\ H_2
		\end{bmatrix}= - \grave{U}_2 \grave{\Sigma}_2 \acute{V}_{22}^H.
	\end{equation*}
\end{proof}

\section{Perturbation Analysis of the RBTLSE Solution}\label{sec5}
In numerical linear algebra, understanding how small perturbations in input data affect the computed solutions is essential for assessing solution reliability. This section presents a mathematical framework for quantifying the sensitivity of RBTLSE solutions through condition number analysis.

We begin by defining a perturbation model for the RBTLSE problem. Consider the original problem matrices 
$A \in \QR^{m \times n}$, $B \in \QR^{m \times d}$, $C \in \QR^{p \times n}$, and $D \in \QR^{p \times d}$.
The perturbed matrices are expressed as
\begin{eqnarray}
	&\widehat{A} =A+\Delta A  \in \QR^{m \times n}, ~  \widehat{B} = B+\Delta B  \in \QR^{m \times d}, \label{eq5.1} \\
	&\widehat{C} = C+\Delta C  \in \QR^{p \times n}, ~  \widehat{D} = D+\Delta D  \in \QR^{p \times d}, \label{eq5.2}
\end{eqnarray}	
where perturbation matrices are expressed as
\begin{eqnarray*}
	&\Delta A=\Delta A_0+\Delta A_1 \i +\Delta A_2 \j + \Delta A_3 \k =\Delta M_1 + \Delta M_2 \j, \\
	&\Delta B=\Delta B_0+\Delta B_1 \i +\Delta B_2 \j + \Delta B_3 \k = \Delta N_1 + \Delta N_2 \j,  \\
	& \Delta C=\Delta C_0+\Delta C_1 \i +\Delta C_2 \j + \Delta C_3 \k = \Delta R_1 + \Delta R_2 \j,  \\ 
	&\Delta D=\Delta D_0+\Delta D_1 \i +\Delta D_2 \j + \Delta D_3 \k = \Delta S_1 + \Delta S_2 \j.
\end{eqnarray*}
To facilitate analysis, we define the concatenated matrices
\begin{equation}\label{eq5.3}
J = \begin{bmatrix}
		C \\
		A
	\end{bmatrix}, ~ K = \begin{bmatrix}
	D \\
	B
\end{bmatrix}, ~ \Delta J = \begin{bmatrix}
\Delta C \\
\Delta A
\end{bmatrix}, ~ \Delta K = \begin{bmatrix}
\Delta D \\
\Delta B
\end{bmatrix}.
\end{equation}
We assume that the perturbation in input matrices is sufficiently small. When the input data is perturbed, the corresponding perturbed RBTLSE problem is given by
\begin{equation}\label{eq5.5}
	\min_{X,\bar{I},\bar{J}}\norm{[\bar{I}, \bar{J}]}_F ~ \textrm{subject to} ~  (\widehat{A}+\bar{I})X = \widehat{B}+\bar{J}, ~ \widehat{C}X=\widehat{D}. 
\end{equation}

We will first study the perturbation analysis of real RBTLSE solution. Let $X_{RT}$ be the real RBTLSE solution to \eqref{eq3.3} and $\widehat{X}_{RT}$ be the real RBTLSE solution to the perturbed problem \eqref{eq5.5}. The sensitivity of $X_{RT}$ to perturbations is quantified through the relative normwise condition number, defined as:
\begin{equation}\label{eq5.6}
	\mathit{k}^{rel}_r(X_{RT}, J, K) = \lim \limits_{\epsilon \rightarrow 0} \sup \left\{\frac{\|\Delta X_{RT}\|_F}{\epsilon \|X_{RT}\|_F} \; | \; \|[\Delta J, \Delta K]\|_F \leq \epsilon \|[J, K]\|_F\right\},
\end{equation} 
where $\Delta X_{RT}= X_{RT}-\widehat{X}_{RT}$. To facilitate analysis, we define the following concatenated matrices in the real domain
\begin{equation}\label{eq5.7}
	J_r = \begin{bmatrix}
	C^R_c \\
		A^R_c
	\end{bmatrix}, ~ K_r = \begin{bmatrix}
		D^R_c \\
		B^R_c
	\end{bmatrix}, ~ \Delta J_r = \begin{bmatrix}
	(\Delta C)^R_c \\
	(\Delta A)^R_c
\end{bmatrix}, ~ \Delta K_r = \begin{bmatrix}
(\Delta D)^R_c \\
(\Delta B)^R_c
\end{bmatrix}.
\end{equation}
\begin{theorem} \label{thm5.1}
	Consider the RBTLSE problem \eqref{eq3.3} and the perturbed RBTLSE problem \eqref{eq5.5}, with the notations introduced in equations \eqref{eq5.3}, \eqref{eq5.7}, and Theorem \ref{thm3.2}. Assume that the conditions specified in Theorem \ref{thm3.2} are satisfied. Let $\hat{S}= \hat{U}_s \hat{S}_s \hat{V}_s^T$ be the skinny SVD of $\hat{S}$. Define
	\begin{equation*}
		Q_r=\begin{bmatrix}
			- \left(\hat{P} \hat{S}^{\dagger}\right)^T \\ I_{4m}
		\end{bmatrix}, ~ S_r = \begin{bmatrix}
		\hat{S}_s & 0 \\
		0 & \hat{\Sigma}_1
	\end{bmatrix},~
	 W_r = \begin{bmatrix}
			\hat{V}_s & \check{V}_1
			\end{bmatrix} =  \begin{blockarray}{c@{}c@{\hspace{4pt}}cl}
			&  \mLabel{n}  & &\\
			\begin{block}{[c@{\hspace{5pt}}c@{\hspace{5pt}}c]l}
				&W_{1r}& & \mLabel{n} \\
				&W_{2r}& & \mLabel{d} \\
			\end{block}
		\end{blockarray}.
	\end{equation*} Then, the relative normwise condition number of the real RBTLSE solution $X_{RT}$ is expressed as
	\begin{equation}\label{eq5.9}
		\mathit{k}^{rel}_{r}(X_{RT},J,K)=\|H_r G_r Z_r\|_2 \frac{\|[J_r, K_r]\|_F}{\|X_{RT}\|_F},
	\end{equation}
	where
	\begin{eqnarray*}
		H_r &=&  \left(\check{V}_{22}^{-T} \otimes W_{1r}^{-T}\right) \Pi_{(d,n)}, \\
		G_r &=& \left(\left(S_r^2 \otimes I_d\right) - \begin{bmatrix}
			0_{4p} &  0\\
			0      & I_{n-4p}
		\end{bmatrix} \otimes (\hat{\Sigma}_2^T \hat{\Sigma}_2)\right)^{-1} \left[I_n \otimes \hat{\Sigma}_2^T, \; S_r \otimes I_d\right], \\
		Z_r &=& \diag \left(\begin{bmatrix}
			0_{4p}  & 0 \\
			0 & I_{n-4p}
		\end{bmatrix} \otimes \left(\hat{U}_2^T Q_r^T\right), \; \begin{bmatrix} 
	I_{4p} & 0 \\
	-\hat{U}_1^T (\hat{P}\hat{S}^{\dagger})\hat{U}_s & I_{n-4p}
\end{bmatrix} \otimes I_d\right).
	\end{eqnarray*}
\end{theorem}
\begin{proof}
The proof follows a logical sequence establishing equivalence between the RBTLSE problem and the real TLSE problem. According to Theorem \ref{thm3.1}, a solution $X_{RT}$ qualifies as a real RBTLSE solution if and only if it also solves the associated real TLSE problem. 
	
	The perturbed real TLSE problem corresponding to \eqref{eq5.5} is formulated as
	\begin{equation}\label{eq5.591}
		\min_{X,\tilde{I},\tilde{J}}\|[\tilde{I}, \tilde{J}]\|_F ~ \text{subject to} ~ (\widehat{A}^R_c+\tilde{I})X = \widehat{B}^R_c+\tilde{J}, ~ \widehat{C}^R_cX=\widehat{D}^R_c.
	\end{equation}
	
	Again applying Theorem \ref{thm3.1}, we know that $\widehat{X}_{RT}$ is a real TLSE solution of the perturbed problem \eqref{eq5.591}. 
	
	A key insight comes from the properties of RB matrices and their Frobenius norm representations established in equation \eqref{eq2.211}. Specifically, we have
	\begin{equation}\label{eq5.597}
		\|[\Delta J, \Delta K]\|_F = \|[\Delta J_r, \Delta K_r]\|_F, ~ \|[J, K]\|_F = \|[J_r, K_r]\|_F.
	\end{equation}
	
	These equalities are critical as they allow us to transform the perturbation analysis of the real RBTLSE solution to the perturbation analysis of the real TLSE solution of the corresponding problem. 
	
	Consequently, we can express the relative normwise condition number of the real RBTLSE solution as
	\begin{equation*}
		\mathit{k}^{rel}_{r}(X_{RT},J,K) = \lim_{\epsilon \rightarrow 0} \sup \left\{\frac{\|\Delta X_{RT}\|_F}{\epsilon \|X_{RT}\|_F} \; \Bigg| \; \|[\Delta J_r, \Delta K_r]\|_F \leq \epsilon \|[J_r, K_r]\|_F\right\}.
	\end{equation*}
	
	This expression is precisely the relative normwise condition number for the real TLSE solution of the corresponding real TLSE problem formulated in equation \eqref{eq3.4}. The equivalence allows us to shift our analysis from the RB domain to the more tractable real domain.
	
	Therefore, to determine the sensitivity of the real RBTLSE solution to perturbations, we can directly apply the established perturbation analysis framework for real TLSE solutions. By invoking Theorem $4.2$ from \cite{MR4371085}, which provides a closed-form expression for the condition number of TLSE problems, we obtain the desired expression for $\mathit{k}^{rel}_{r}(X_{RT},J,K)$.
\end{proof}
To quantify the impact of perturbations on the real RBTLSE solution, we define the relative size of input perturbations as
\begin{equation}\label{eq5.91}
	\varepsilon_n = \frac{\|[\Delta J, \, \Delta K]\|_F}{\|[J, \, K]\|_F}.
	\end{equation}
Using the condition number derived previously, we can bound the relative error in the real RBTLSE solution as follows
\begin{equation}\label{eq5.10}
	\frac{\| \Delta X_{RT} \|_F}{\|X_{RT} \|_F} \leq \mathit{k}^{rel}_{r}(X_{RT},J,K) \varepsilon_n \equiv U_r.
\end{equation}
Here, $U_r$ denotes the first-order upper bound on the relative forward error. This inequality provides a practical tool for estimating the worst-case impact of data perturbations on solution accuracy.

Having analyzed the sensitivity of real RBTLSE solutions to perturbations, we now extend our analysis to complex RBTLSE solutions. For the complex RBTLSE problem formulated in equation \eqref{eq4.3}, we examine how perturbations in the input matrices $A$, $B$, $C$, and $D$ affect the complex RBTLSE solution.

Denote $X_{CT}$ as the complex RBTLSE solution of the original problem. Let $\widehat{X}_{CT}$ represent the solution corresponding to the perturbed system \eqref{eq5.5}, and define the associated error by $\Delta X_{CT} = \widehat{X}_{CT} - X_{CT}$. The sensitivity of $X_{CT}$ to perturbations is quantified through the relative normwise condition number, defined as
\begin{equation}\label{eq5.11}
	\mathit{k}^{rel}_c(X_{CT}, J, K) = \lim \limits_{\epsilon \rightarrow 0} \sup \left\{\frac{\|\Delta X_{CT}\|_F}{\epsilon \|X_{CT}\|_F} \; | \; \|[\Delta J, \Delta K]\|_F \leq \epsilon \|[J, K]\|_F\right\}.
\end{equation} 
To facilitate our analysis in the complex domain, we define the following concatenated matrices:
\begin{equation}\label{eq5.12}
	J_c = \begin{bmatrix}
		C^C_c \\
		A^C_c
	\end{bmatrix}, ~ K_c = \begin{bmatrix}
		D^C_c \\
		B^C_c
	\end{bmatrix}, ~ \Delta J_c = \begin{bmatrix}
	(\Delta C)^C_c \\
	(\Delta A)^C_c
\end{bmatrix}, ~ \Delta K_c = \begin{bmatrix}
(\Delta D)^C_c \\
(\Delta B)^C_c
\end{bmatrix}.
\end{equation}
\begin{theorem} \label{thm5.2}
	Consider the RBTLSE problem \eqref{eq4.3} and the perturbed RBTLSE problem \eqref{eq5.5}, with the notations introduced in equations \eqref{eq5.3}, \eqref{eq5.12}, and Theorem \ref{thm4.2}. Assume that the conditions specified in Theorem \ref{thm4.2} are satisfied. Let $\grave{S}= \grave{U}_s \grave{S}_s \grave{V}_s^H$ be the skinny SVD of $\grave{S}$. Define
	\begin{equation*}
		Q_c=\begin{bmatrix}
			- \left(\grave{P} \grave{S}^{\dagger}\right)^H \\ I_{2m}
		\end{bmatrix}, ~ S_c = \begin{bmatrix}
			\grave{S}_s & 0 \\
			0 & \grave{\Sigma}_1
		\end{bmatrix}, ~
		W_c= \begin{bmatrix}
			\grave{V}_s & \acute{V}_1
		\end{bmatrix} =  \begin{blockarray}{c@{}c@{\hspace{4pt}}cl}
			&  \mLabel{n}  & &\\
			\begin{block}{[c@{\hspace{5pt}}c@{\hspace{5pt}}c]l}
				&W_{1c}& & \mLabel{n} \\
				&W_{2c}& & \mLabel{d} \\
			\end{block}
		\end{blockarray}.
	\end{equation*} Then, the relative normwise condition number of the complex RBTLSE solution $X_{CT}$ is expressed as
	\begin{equation}\label{5.14}
		\mathit{k}^{rel}_{c}(X_{CT},J,K)=\|H_c G_c Z_c\|_2 \frac{\|[J_c, K_c]\|_F}{\|X_{CT}\|_F},
	\end{equation}
	where
	\begin{eqnarray*}
		H_c &=&  \left(\acute{V}_{22}^{-H} \otimes W_{1c}^{-H}\right) \Pi_{(d,n)}, \\
		G_c &=& \left(\left(S_c^2 \otimes I_d\right) - \begin{bmatrix}
			0_{2p} &  0\\
			0      & I_{n-2p}
		\end{bmatrix} \otimes (\grave{\Sigma}_2^H \grave{\Sigma}_2)\right)^{-1} \left[I_n \otimes \grave{\Sigma}_2^H, \; S_c \otimes I_d\right], \\
		Z_c &=& \diag \left(\begin{bmatrix}
			0_{2p}  & 0 \\
			0 & I_{n-2p}
		\end{bmatrix} \otimes \left(\grave{U}_2^H Q_c^H\right), \; \begin{bmatrix} 
			I_{2p} & 0 \\
			-\grave{U}_1^H (\grave{P}\grave{S}^{\dagger})\grave{U}_s & I_{n-2p}
		\end{bmatrix} \otimes I_d\right).
	\end{eqnarray*}
\end{theorem}
\begin{proof}
	The proof follows a similar structure to that of Theorem \ref{thm5.1}.
\end{proof}
The first-order upper bound on the relative forward error for the complex RBTLSE solution is given by
\begin{equation}\label{eq5.15}
	\frac{\| \Delta X_{CT} \|_F}{\|X_{CT} \|_F} \leq \mathit{k}^{rel}_{c}(X_{CT},J,K) \varepsilon_n \equiv U_c.
\end{equation}
Here, $U_c$ denotes the first-order upper bound on the relative error for the complex solution. This inequality provides a practical tool for estimating the maximum impact of data perturbations on complex solution accuracy.

\section{Numerical Verification}\label{sec6}
This section develops numerical algorithms for obtaining special solutions to the RBTLSE problems, drawing on the concepts discussed earlier. Numerical examples are also included to demonstrate the validity of the proposed methods.
\begin{algorithm} [H]
	\caption{For real solution of RBTLSE problem}\label{alg3}
	\textbf{Input:} $A= A_0+ A_1\i+ A_2\j+ A_3\k \in \QR^{m \times n}$, $B= B_0+ B_1\i+ B_2\j+ B_3\k \in \QR^{m \times d}$, $C= C_0+ C_1\i+ C_2\j+ C_3\k \in  \QR^{p \times n}$, and $D= D_0+ D_1\i+ D_2\j+ D_3\k \in \QR^{p \times d}.$ \\
	\textbf{Assumptions:}  $C^R_c$ possesses full row rank and $m \geq n+d$.\\
	\textbf{Output:} Real solution $X_{RT}$ to the RBTLSE problem.
	\begin{enumerate}
		\item Construct the augmented matrices $\hat{P} = [A^R_c, B^R_c]$ and $\hat{S} = [C^R_c, D^R_c]$.
		\item Compute the QR factorization of $\hat{S}^T$ as specified in equation \eqref{eqth1}.
		\item Perform thin SVD of matrix $\hat{P}\hat{Q}_2$ as defined in equation \eqref{eq23}.
		\item Calculate the partitioned matrix $\check{V}$ according to equation \eqref{eq24.1}.
		\item Check whether $\hat{\sigma}_{n-4p} > \hat{\sigma}_{n-4p+1}$ and verify that $\check{V}_{22}$ is invertible; if so, compute the real RBTLSE solution using equation \eqref{eq26}.
\end{enumerate}
	\vspace{-0.2cm}
\end{algorithm}

\begin{algorithm} [H]
	\caption{For complex solution of RBTLSE problem}\label{alg1}
	\textbf{Input:} $A= M_1+ M_2\j \in \QR^{m \times n}$, $B= N_1+ N_2\j \in \QR^{m \times d}$, $C= R_1+ R_2\j \in  \QR^{p \times n}$, and $D= S_1+ S_2\j \in \QR^{p \times d}.$ \\
	\textbf{Assumptions:} $C^C_c$ possesses full row rank and $m \geq n+d$.\\
	\textbf{Output:} Complex solution $X_{CT}$ to the RBTLSE problem.
	\begin{enumerate}
		\item Construct the augmented matrices $\grave{P} = [A^C_c, B^C_c]$ and $\grave{S} = [C^C_c, D^C_c]$.
		\item Compute the QR factorization of $\grave{S}^H$ as specified in equation \eqref{4.15}.
		\item Perform thin SVD of matrix $\grave{P}\grave{Q}_2$ as defined in equation \eqref{eq4.16}.
		\item Calculate the partitioned matrix $\acute{V}$ according to equation \eqref{eq4.19}.
		\item Check whether $\grave{\sigma}_{n-2p} > \grave{\sigma}_{n-2p+1}$ and verify that $\acute{V}_{22}$ is invertible; if so, compute the complex RBTLSE solution using equation \eqref{comsol}.
     \end{enumerate}
	\vspace{-0.2cm}
\end{algorithm}
The following examples illustrate the performance of the developed algorithms. All experiments were conducted in MATLAB $R2021b$ on a machine with a $3.00$ GHz Intel Core $i7-9700$ processor and $16$ GB RAM.

 \begin{example} \label{ex2}
Let \( A, B, C, D \in \QR \) be reduced biquaternion matrices of sizes \( m \times n \), \( m \times d \), \( p \times n \), and \( p \times d \), respectively, with each component matrix generated randomly from \( \mathbb{R} \) using \texttt{randn}(\(\cdot\)). That is, each matrix is of the form \( Z = Z_0 + Z_1\i + Z_2\j + Z_3\k \) for appropriate dimensions.

We set dimensions as $m=30t$, $n=10t$, $p=2t$, and $d=2$, where $t$ is a positive integer parameter. Using these matrices as inputs to Algorithm \ref{alg3}, we calculate the real RBTLSE solution $X_{RT}$ to \eqref{eq3.3}.
	
	For the case where $t=1$, we obtain
	\begin{equation*}
		X_{RT} = \begin{bmatrix}
			5.6008 & -5.6050 \\
			2.9891 & -2.1700 \\
			-5.5439 & 4.7055 \\
			0.0606 & -8.8370 \\
			2.1532 & -0.4639 \\
			-0.6891 & 0.7625 \\
			2.2691 & -0.0285 \\
			1.1575 & -1.3651 \\
			-0.0055 & 0.0812 \\
			-0.0648 & 0.2982
		\end{bmatrix}.
	\end{equation*}
	To evaluate the accuracy of our solution, we compute two error metrics:
	\begin{equation*}
		\epsilon_{r1}=\norm{(A+\bar{E})X_{RT}-(B+\bar{F})}_F,  ~ \epsilon_{r2}=\norm{CX_{RT}-D}_F,
	\end{equation*}
	where $\bar{E}$ and $\bar{F}$ are defined as in equation \eqref{eq3.3}. The error values for different values of $t$ are presented in Table \ref{tab21}. The consistently small error values (all less than $10^{-13}$) across various dimensions demonstrate the effectiveness of Algorithm \ref{alg3}  in finding real RBTLSE solution.
\end{example}
\begin{table}[H]
	\begin{center}
		\begin{tabular}{ c  c c} 
			\toprule
			$t$ &$\epsilon_{r1}$  & $\epsilon_{r2}$\\
			\midrule
			$1$    & $1.0391 \times 10^{-13}$&            $9.5143 \times 10^{-15} $     \\
			$3$  & $3.3531 \times 10^{-13}$ &            $4.0017 \times 10^{-14} $        \\ 
			$5$  & $2.4956 \times 10^{-13}$& $4.7754 \times 10^{-14} $        \\ 
			$7$   & $2.9334\times 10^{-13}$&            $4.3344 \times 10^{-14} $          \\ 
			$9$  & $4.8715 \times 10^{-13}$               & $7.1455 \times 10^{-14} $       \\
			\bottomrule
		\end{tabular}
		\caption{Computational accuracy of Algorithm \ref{alg3}.} \label{tab21}
	\end{center}
\end{table}

\begin{example}\label{ex2.1}
	Let us examine the RBTLSE problem \eqref{eq3.3} with coefficient matrices $A$, $B$, $C$, $D$ as specified in Example \ref{ex2}. We investigate the sensitivity of the solution by introducing random perturbations $\Delta A$, $\Delta B$, $\Delta C$, and $\Delta D$ to these matrices. Our objective is to analyze how these perturbations affect the real RBTLSE solution $X_{RT}$. We quantify the exact relative forward error as $\frac{\|\Delta X_{RT}\|_F}{\|X_{RT}\|_F}$ and compare it with the theoretical upper bound $U_r$ derived from equation \eqref{eq5.10}.
	
	To systematically evaluate the error behavior, Table \ref{tab2.1} presents three key metrics: the magnitude of perturbations $\|[\Delta J, \Delta K]\|_F$, the resulting relative forward error $\frac{\|\Delta X_{RT}\|_F}{\|X_{RT}\|_F}$, and the theoretical upper bound $U_r$. These metrics are computed for various values of parameter $t$ and different perturbation magnitudes to provide comprehensive insight into the solution's sensitivity.
\end{example}
\begin{table}[h]
	\centering
	\small
	\begin{tabular}{c@{\hskip 1.2cm}c@{\hskip 1.2cm}c@{\hskip 1.2cm}c}
		\toprule
		$t$ & $\|[\Delta J, \Delta K]\|_F$ & $\frac{\|\Delta X_{RT}\|_F}{\|X_{RT}\|_F}$ & $U_r$ \\
		\midrule
		1  & $2.2703\text{e}{-11}$ & $4.6231\text{e}{-12}$ & $8.1719\text{e}{-11}$ \\
		& $5.8686\text{e}{-09}$ & $6.3442\text{e}{-09}$ & $1.0240\text{e}{-07}$ \\
		& $1.3514\text{e}{-07}$ & $1.1387\text{e}{-06}$ & $8.1158\text{e}{-06}$ \\
		\midrule
		3  & $8.1092\text{e}{-10}$ & $1.3748\text{e}{-09}$ & $1.4113\text{e}{-08}$ \\
		& $1.1979\text{e}{-08}$ & $1.7344\text{e}{-07}$ & $3.2495\text{e}{-06}$ \\
		& $4.2278\text{e}{-06}$ & $1.4833\text{e}{-06}$ & $8.6862\text{e}{-05}$ \\
		\midrule
		5  & $4.9506\text{e}{-12}$ & $2.5434\text{e}{-12}$ & $8.4246\text{e}{-10}$ \\
		& $1.2937\text{e}{-09}$ & $1.7800\text{e}{-09}$ & $2.3038\text{e}{-08}$ \\
		& $2.7633\text{e}{-05}$ & $6.0400\text{e}{-05}$ & $4.9883\text{e}{-04}$ \\
		\midrule
		7  & $2.0329\text{e}{-11}$ & $3.0678\text{e}{-11}$ & $1.0081\text{e}{-09}$ \\
		& $4.0097\text{e}{-10}$ & $1.7030\text{e}{-09}$ & $2.4548\text{e}{-08}$ \\
		& $3.6009\text{e}{-06}$ & $5.9840\text{e}{-06}$ & $4.1659\text{e}{-04}$ \\
		\midrule
		9  & $2.3034\text{e}{-11}$ & $1.4516\text{e}{-11}$ & $5.3757\text{e}{-10}$ \\
		& $2.5079\text{e}{-07}$ & $2.5897\text{e}{-07}$ & $7.0505\text{e}{-06}$ \\
		& $1.0170\text{e}{-05}$ & $5.3356\text{e}{-06}$ & $2.2537\text{e}{-04}$ \\
		\bottomrule
	\end{tabular}
	\caption{Comparison of relative forward error and upper bounds for the perturbed problem.}
	\label{tab2.1}
\end{table}
Table \ref{tab2.1} illustrates that the actual relative forward errors consistently remain below their theoretical upper bounds $U_r$ regardless of perturbation size or parameter value. This empirical evidence confirms the validity of our error bound analysis and its practical utility in predicting real RBTLSE solution sensitivity. 

\begin{example} \label{ex1}
Let \( A \in \QR^{m \times n} \), \( B \in \QR^{m \times d} \), \( C \in \QR^{p \times n} \), and \( D \in \QR^{p \times d} \) be reduced biquaternion matrices, where each matrix is of the form \( Z = Z_1 + Z_2\j \), with \( Z_1, Z_2 \in \C^{s \times t} \) having complex entries generated using \texttt{rand}(\(\cdot\)) for appropriate dimensions.

We set dimensions as $m=50t$, $n=6t$, $p=2t$, and $d=3$, where $t$ is a positive integer parameter. Using these matrices as inputs to Algorithms \ref{alg1}, we calculate the complex RBTLSE solution $X_{CT}$ to \eqref{eq4.3}.

For the case where $t=1$, we obtain
	\begin{equation*}
	X_{CT} = \begin{bmatrix}
		2.5570-0.1634 \i & -1.0157+0.1433\i & 0.1109+0.9360\i \\
		0.2996-1.8524\i & 0.7409-1.6943\i & -0.9319-0.2281\i \\
		0.6571-0.2063\i & -0.8987+0.9218\i & 0.6333+0.4526\i \\
		0.3458+1.0246\i & 0.8121+0.3662\i & 0.0481-0.5039\i \\
		-2.1026+0.7757\i & 0.1158-0.4280\i & -0.0486+0.2270\i \\
		-0.8987+0.8807\i & 1.6070+0.8080\i & 1.2279-1.0744\i
	\end{bmatrix}.
\end{equation*}
To evaluate the accuracy of our solution, we compute two error metrics:	
\begin{equation*}
\epsilon_{c1}=\norm{(A+\bar{G})X_{CT}-(B+\bar{H})}_F, ~ \epsilon_{c2}=\norm{CX_{CT}-D}_F,
\end{equation*}
	where $\bar{G}$ and $\bar{H}$ are defined as in equation \eqref{eq4.3}. The error values for different values of $t$ are presented in Table \ref{tab11}.  The consistently small error values (all less than $10^{-13}$) across various dimensions demonstrate the effectiveness of Algorithm \ref{alg1} in finding complex solutions to the RBTLSE problem.
\end{example}
\begin{table}[H]
	\begin{center}
		\begin{tabular}{ c  c c} 
			\toprule
			$t$ &$\epsilon_{c1} $  & $\epsilon_{c2}$\\
			\midrule
			$1$     & $3.4351 \times 10^{-14}$&            $4.1297 \times 10^{-15} $     \\
			$3$    & $2.5537 \times 10^{-13}$ &            $2.2070 \times 10^{-14} $        \\ 
			$5$  & $9.2360 \times 10^{-14}$& $1.0560 \times 10^{-14} $        \\ 
			$7$    & $5.7324 \times 10^{-14}$&            $5.1796 \times 10^{-14} $          \\ 
			$9$   & $6.4804 \times 10^{-13}$               & $1.2717 \times 10^{-13} $       \\
			\bottomrule
		\end{tabular}
		\caption{Computational accuracy of Algorithm \ref{alg1}.} \label{tab11}
	\end{center}
\end{table}

\begin{example}\label{ex1.1}
	Let us examine the RBTLSE problem \eqref{eq4.3} with coefficient matrices $A$, $B$, $C$, $D$ as specified in Example \ref{ex1}. We investigate the sensitivity of the solution by introducing random perturbations $\Delta A$, $\Delta B$, $\Delta C$, and $\Delta D$ to these matrices. Our objective is to analyze how these perturbations affect the complex RBTLSE solution $X_{CT}$. We quantify the exact relative forward error as $\frac{\|\Delta X_{CT}\|_F}{\|X_{CT}\|_F}$ and compare it with the theoretical upper bound $U_c$ derived from equation \eqref{eq5.15}.
	
	To systematically evaluate the error behavior, Table \ref{tab1.11} presents three key metrics: the magnitude of perturbations $\|[\Delta J, \Delta K]\|_F$, the resulting relative forward error $\frac{\|\Delta X_{CT}\|_F}{\|X_{CT}\|_F}$, and the theoretical upper bound $U_c$. These metrics are computed for various values of parameter $t$ and different perturbation magnitudes to provide comprehensive insight into the solution's sensitivity.
\end{example}
\begin{table}[h]
	\centering
	\small
	\begin{tabular}{c@{\hskip 1.2cm}c@{\hskip 1.2cm}c@{\hskip 1.2cm}c}
		\toprule
		$t$ & $\|[\Delta J, \Delta K]\|_F$ & $\frac{\|\Delta X_{CT}\|_F}{\|X_{CT}\|_F}$ & $U_c$ \\
		\midrule
		1  & $2.4689\text{e}{-14}$ & $3.0408\text{e}{-14}$ & $2.6618\text{e}{-13}$ \\
		& $1.2245\text{e}{-09}$ & $6.9055\text{e}{-10}$ & $7.4350\text{e}{-09}$ \\
		& $1.5817\text{e}{-05}$ & $1.6792\text{e}{-05}$ & $1.5182\text{e}{-04}$ \\
		\midrule
		3  & $1.7399\text{e}{-13}$ & $2.8634\text{e}{-13}$ & $5.1241\text{e}{-12}$ \\
		& $3.0086\text{e}{-10}$ & $8.1548\text{e}{-11}$ & $3.5047\text{e}{-09}$ \\
		& $2.4039\text{e}{-06}$ & $5.5996\text{e}{-06}$ & $2.2141\text{e}{-05}$ \\
		\midrule
		5  & $3.0605\text{e}{-12}$ & $1.0771\text{e}{-11}$ & $8.6044\text{e}{-11}$ \\
		& $2.2956\text{e}{-09}$ & $5.2550\text{e}{-10}$ & $9.4879\text{e}{-09}$ \\
		& $5.0140\text{e}{-06}$ & $1.0008\text{e}{-06}$ & $3.6651\text{e}{-05}$ \\
		\midrule
		7  & $1.8635\text{e}{-13}$ & $1.3509\text{e}{-13}$ & $1.7363\text{e}{-12}$ \\
		& $1.6230\text{e}{-11}$ & $1.9775\text{e}{-11}$ & $2.1878\text{e}{-10}$ \\
		& $1.8778\text{e}{-07}$ & $1.9025\text{e}{-07}$ & $4.1458\text{e}{-06}$ \\
		\midrule
		9  & $4.9001\text{e}{-12}$ & $1.6812\text{e}{-11}$ & $8.7683\text{e}{-11}$ \\
		& $1.8008\text{e}{-09}$ & $1.9193\text{e}{-09}$ & $4.3635\text{e}{-08}$ \\
		& $2.9757\text{e}{-07}$ & $1.5750\text{e}{-07}$ & $5.1066\text{e}{-06}$ \\
		\bottomrule
	\end{tabular}
	\caption{Comparison of relative forward error and upper bounds for the perturbed problem.}
	\label{tab1.11}
\end{table}
Table \ref{tab1.11} illustrates that the actual relative forward errors consistently remain below their theoretical upper bounds $U_c$ regardless of perturbation size or parameter value. This empirical evidence confirms the validity of our error bound analysis and its practical utility in predicting complex RBTLSE solution sensitivity. 

\begin{example} \label{ex5.4}
	This example illustrates the performance of the proposed RBTLSE method compared with the existing RBLSE method~\cite{ahmad2024solutions} in solving inconsistent equality-constrained reduced biquaternion matrix equations (RBMEs). 
	
	Specifically, we generate exact real and complex solutions to consistent equality-constrained RBMEs. Perturbations are then introduced to simulate inconsistencies, and we compute approximate solutions using both RBTLSE and RBLSE methods. The quality of the computed solutions is assessed using the Frobenius norm error with respect to the original exact solution. 
	
	The testing is carried out under two scenarios for each setting:  
	\begin{itemize}
	\item perturbations in both the coefficient and right-hand side matrices, and  
	\item perturbations only in the right-hand side matrix.  
	\end{itemize}
	The errors are recorded for increasing problem sizes to compare the performance of the two methods.
	\begin{enumerate}
		\item[\rm (i)] \textbf{Real solutions:}
		
		Let $E= E_0+E_1\i+E_2\j+E_3\k \in \QR^{m \times 50}$ with $(m>50)$, where each component matrix is defined as $E_i= \texttt{randn}(m,50) \in \R^{m \times 50}$. Similarly, define $C= C_0+C_1\i+C_2\j+C_3\k \in \QR^{10 \times 50}$, where $C_i= \texttt{randn}(10,50) \in \R^{10 \times 50}$ for $i=0,1,2,3$. Let $X_1$ be a randomly generated real matrix given by $X_1 = \texttt{randn}(50,35) \in \R^{50 \times 35}$.
		
		Using these matrices, we construct $F=E_0X_1+E_1X_1\i+E_2X_1\j+E_3X_1\k$ and $D=C_0X_1+C_1X_1\i+C_2X_1\j+C_3X_1\k$. Clearly, the equality-constrained RBME
		\begin{equation*}
			EX = F, ~ \textrm{subject to} ~ CX=D
		\end{equation*}
		is consistent, with $X_1$ being its exact real solution.
		
		To evaluate solution accuracy under inconsistencies, we perturb $E$ and $F$ to obtain:
		\[
		AX \approx B, ~ \text{subject to} ~ CX = D,
		\]
		where $A = E + dE$, $B = F + dF$, and $X \in \R^{50 \times 35}$.
		
		\textbf{Case 1:} Perturbations in both $E$ and $F$.
		
		Generate $G = \texttt{rand}(85,85)$ and set $H = 0.01 \cdot \texttt{rand}(m,85) \cdot G$. Define
		\begin{align*}
			dE &= H(:,1\!:\!50) + H(:,1\!:\!50)\i + H(:,1\!:\!50)\j + H(:,1\!:\!50)\k, \\
			dF &= H(:,51\!:\!85) + H(:,51\!:\!85)\i + H(:,51\!:\!85)\j + H(:,51\!:\!85)\k.
		\end{align*}
		
		\textbf{Case 2:} Perturbation only in $F$.
		
		Generate $G = \texttt{rand}(35,35)$ and set $H = 0.01 \cdot \texttt{rand}(m,35) \cdot G$. Define
		\begin{align*}
			dE &= 0, \\
			dF &= H + H\i + H\j + H\k.
		\end{align*}
		
		In both cases, we compute approximate real solutions $X_{RT}$ and $X_{RL}$ using the RBTLSE (Algorithm~\ref{alg3}) and RBLSE methods~\cite[Algorithm $1$]{ahmad2024solutions}, respectively.
		
		\item[\rm (ii)] \textbf{Complex solutions:}
		
		Let $M=M_1+M_2\mathbf{j} \in \QR^{m \times 50}$ with $(m>50)$, where each component matrix is defined as $M_i=\texttt{randn}(m,50)+\texttt{randn}(m,50)\mathbf{i}$ for $i=1,2$. Similarly, define $R= R_1+R_2\mathbf{j} \in \QR^{10 \times 50}$, where $R_i=\texttt{randn}(10,50)+\texttt{randn}(10,50)\mathbf{i}$ for $i=1,2$. Let $X_2$ be a randomly generated complex matrix given by $X_2 = \texttt{randn}(50,35)+\texttt{randn}(50,35)\mathbf{i} \in \mathbb{C}^{50 \times 35}$. 
		
		Using these matrices, we construct $N=M_1X_2+M_2X_2 \j$ and $S=R_1X_2+R_2X_2 \j$. Clearly, the equality-constrained RBME
		\begin{equation*}
			MX = N, ~ \textrm{subject to} ~ RX=S
		\end{equation*}
		is consistent, with $X_2$ being its exact complex solution.
		
		Perturb $M$ and $N$ to obtain the inconsistent system:
		\[
		PX \approx Q, ~ \text{subject to} ~ RX = S,
		\]
		where $P = M + dM$, $Q = N + dN$, and $X \in \C^{50 \times 35}$.
		
		\textbf{Case 1:} Perturbations in both $M$ and $N$.
		
		Generate $T = \texttt{rand}(85,85) + \texttt{rand}(85,85)\i$, $I = \texttt{rand}(m,85) + \texttt{rand}(m,85)\i$, and set $J = 0.01 \cdot IT$. Define
		\begin{align*}
			dM &= J(:,1\!:\!50) + J(:,1\!:\!50)\j, \\
			dN &= J(:,51\!:\!85) + J(:,51\!:\!85)\j.
		\end{align*}
		
		\textbf{Case 2:} Perturbation only in $N$.
		
		Generate $T = \texttt{rand}(35,35) + \texttt{rand}(35,35)\i$, $I = \texttt{rand}(m,35) + \texttt{rand}(m,35)\i$, and set $J = 0.01 \cdot IT$. Define
		\begin{align*}
			dM &= 0, \\
			dN &= J + J\j.
		\end{align*}
		
		For both scenarios, we compute the complex solutions $X_{CT}$ and $X_{CL}$ using the RBTLSE (Algorithm~\ref{alg1}) and RBLSE methods~\cite[Algorithm $2$]{ahmad2024solutions}, respectively.
	\end{enumerate}
	
	\textbf{Note:} $X_1$ and $X_2$ represent the most accurate real and complex solutions for the original consistent systems. For evaluation, we define the error metrics:
	\[
	\epsilon_{RT} = \|X_{RT} - X_1\|_F, ~
	\epsilon_{RL} = \|X_{RL} - X_1\|_F, ~
	\epsilon_{CT} = \|X_{CT} - X_2\|_F, ~
	\epsilon_{CL} = \|X_{CL} - X_2\|_F.
	\]
	We report these metrics for varying values of $m$ by averaging results over 20 independent trials for each $m$.
\end{example}

Figure~\ref{fig12} compares the errors $\epsilon_{RT}$ and $\epsilon_{RL}$ across various values of $m$ for both perturbation cases. In \textbf{Case~$\mathbf{1}$}, $\epsilon_{RT}$ is consistently lower than $\epsilon_{RL}$, whereas in \textbf{Case~$\mathbf{2}$}, $\epsilon_{RL}$ yields smaller errors.

Similarly, Figure~\ref{fig11} presents the comparison between $\epsilon_{CT}$ and $\epsilon_{CL}$. For \textbf{Case~$\mathbf{1}$}, $\epsilon_{CT}$ outperforms $\epsilon_{CL}$ across all $m$, while in \textbf{Case~$\mathbf{2}$}, $\epsilon_{CL}$ consistently achieves lower error than $\epsilon_{CT}$.
	\begin{figure}[H]
		\centering
		\subfigure[]{\includegraphics[width=0.51\textwidth]{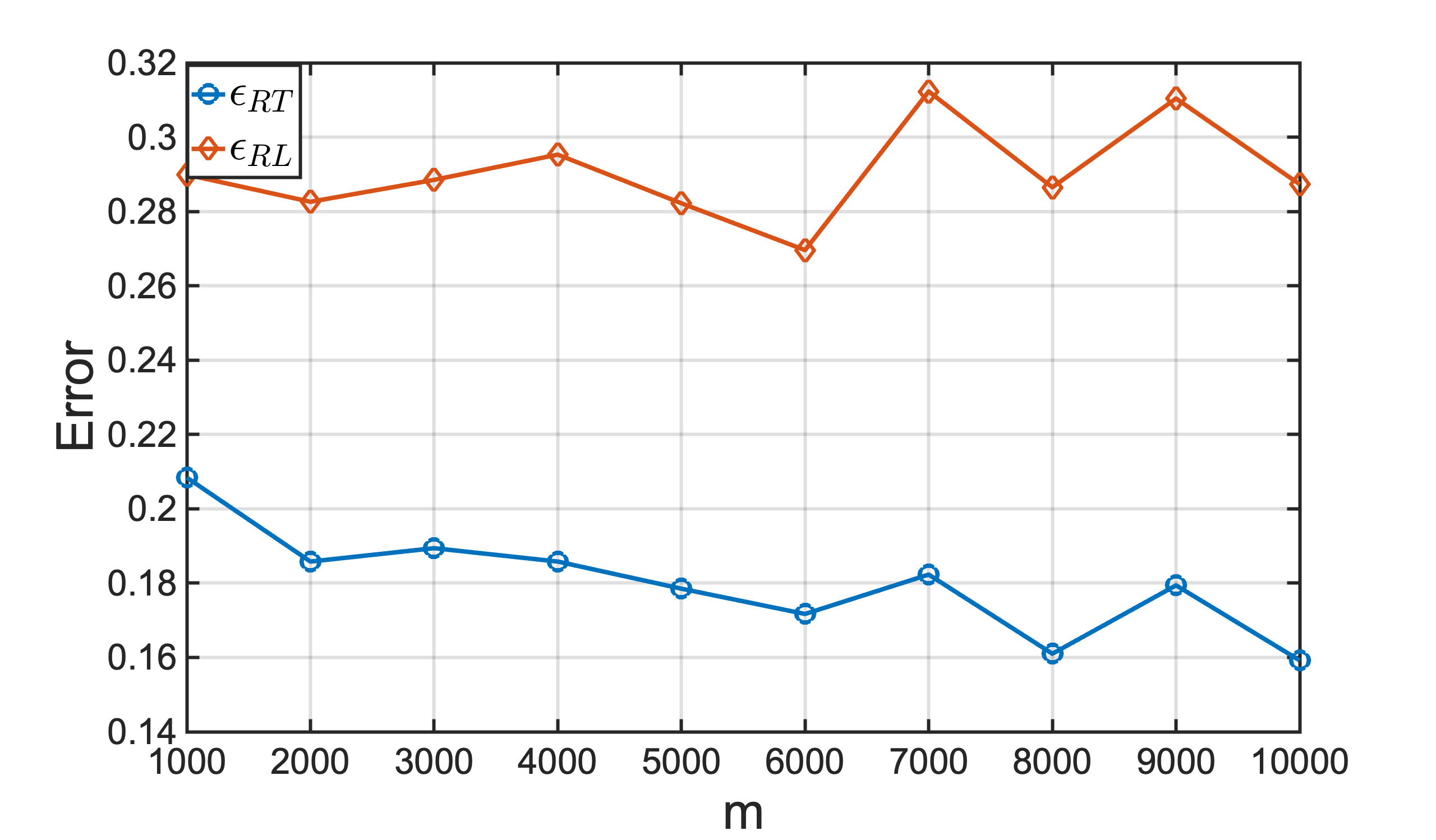}} 
		\hspace{-0.7cm}
		\subfigure[]{\includegraphics[width=0.51\textwidth]{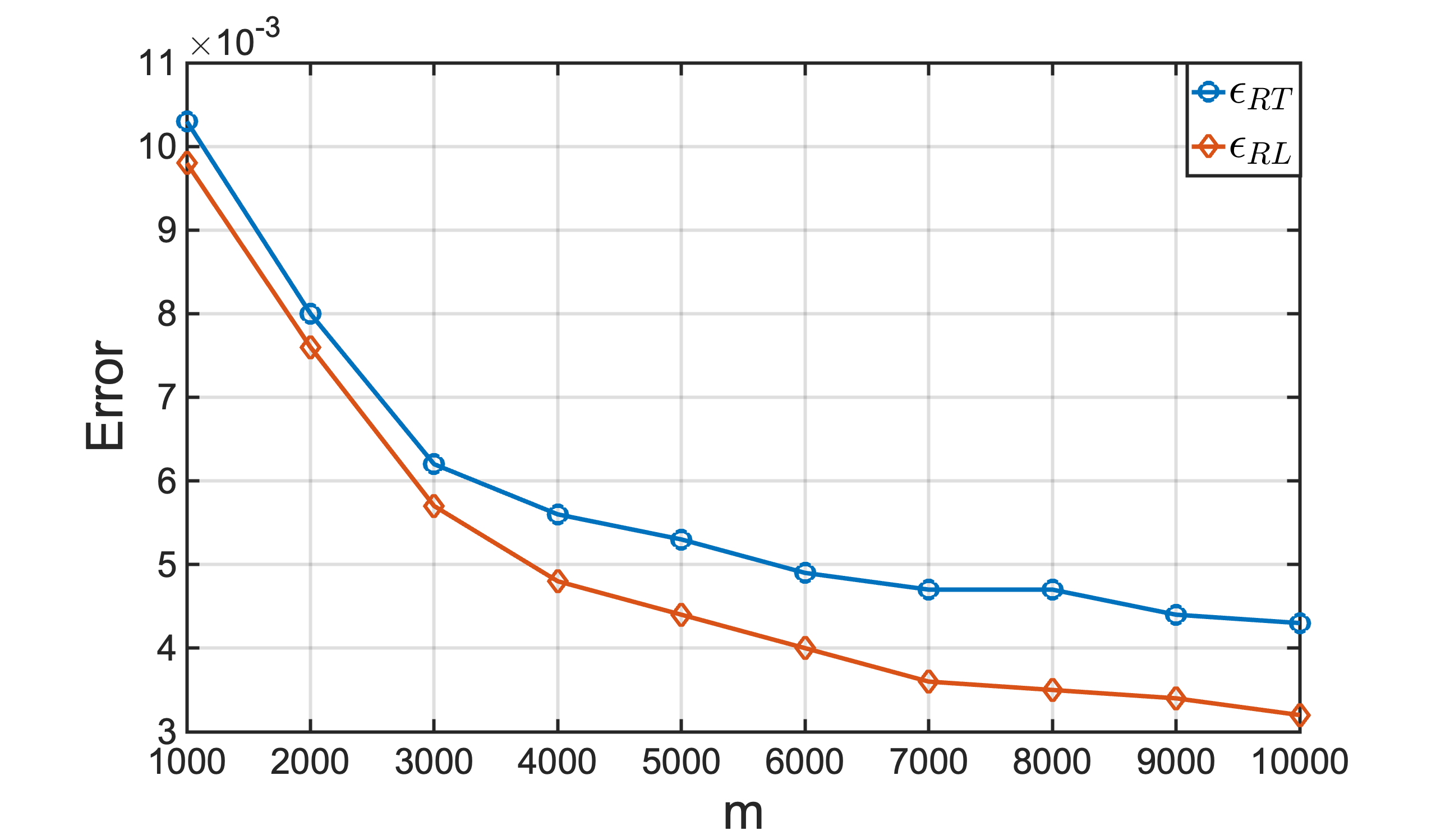}} 
		\caption{Comparison of errors in computing real solutions for the RBTLSE and RBLSE problems (a)  \textbf{Case} $\mathbf{1}$ (b)  \textbf{Case} $\mathbf{2}$}
		\label{fig12}
	\end{figure}

	\begin{figure}[H]
	\centering
	\subfigure[]{\includegraphics[width=0.51\textwidth]{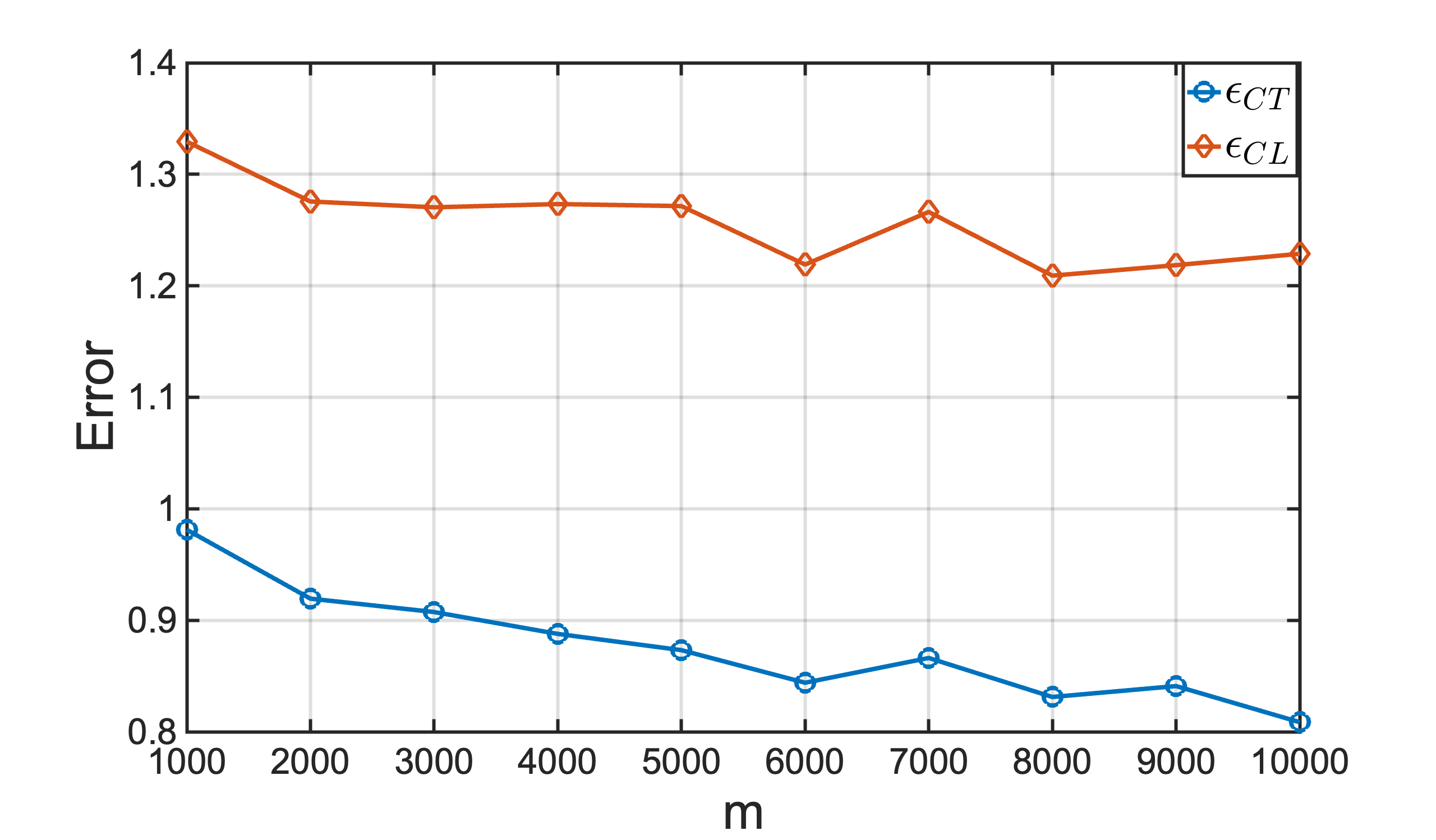}} 
	\hspace{-0.7cm}
	\subfigure[]{\includegraphics[width=0.51\textwidth]{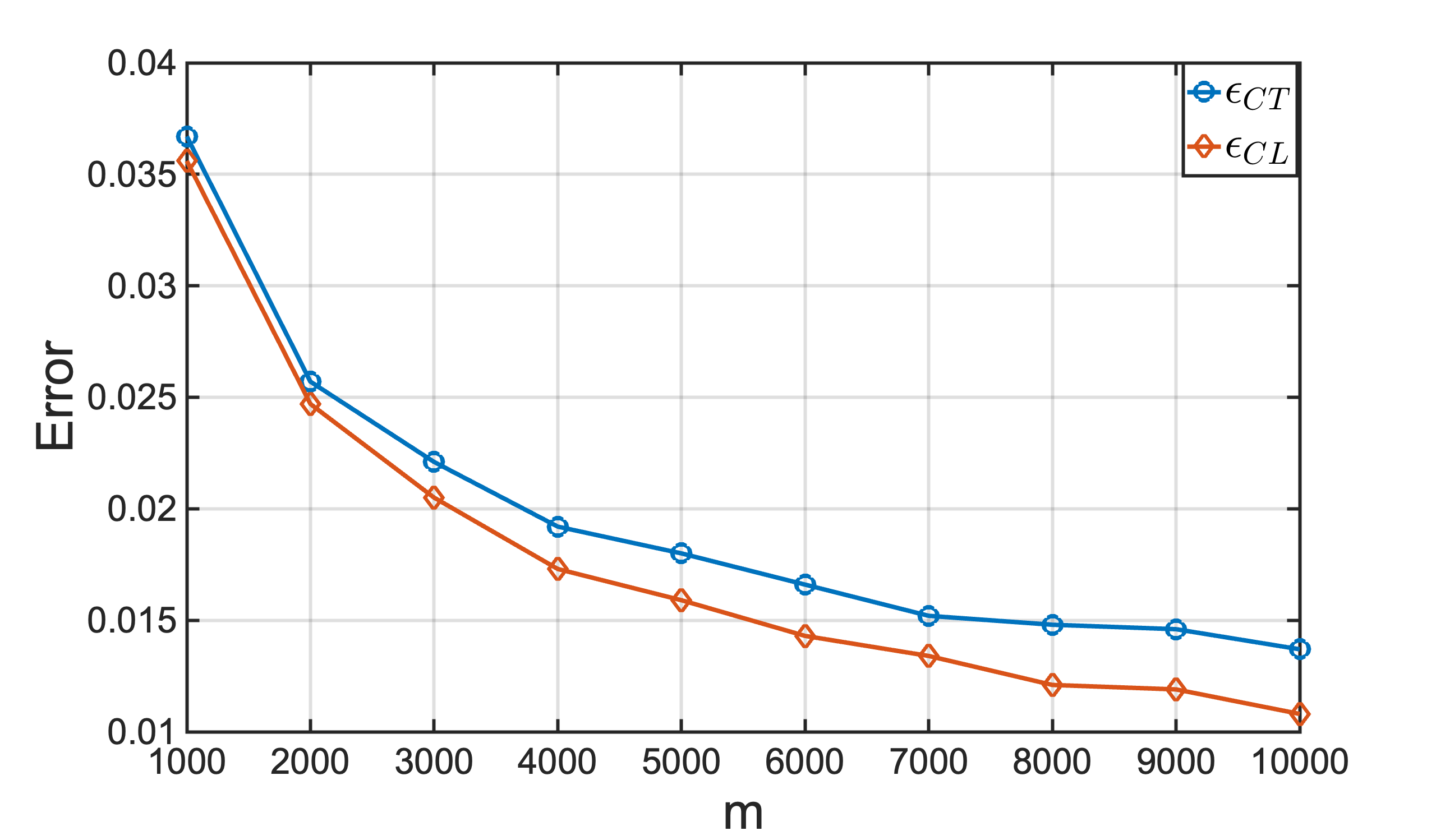}} 
	\caption{Comparison of errors in computing complex solutions for the RBTLSE and RBLSE problems (a)  \textbf{Case} $\mathbf{1}$ (b)  \textbf{Case} $\mathbf{2}$}
	\label{fig11}
\end{figure}
The following remark summarizes the key point from this example:
	\begin{remark}
In the presence of noise, equality-constrained RBMEs may become inconsistent, necessitating approximate solution methods. The RBTLSE approach is better suited for scenarios where errors affect both the coefficient matrix and the right-hand side, whereas the RBLSE method is more effective when errors are confined to the right-hand side. This distinction emphasizes the need to select an appropriate solution strategy based on the type and location of errors.
	\end{remark}

	\section{Conclusions} \label{sec7}
	In this paper, we developed a comprehensive framework for addressing the RBTLSE problem. We identified the conditions under which real and complex solutions can be obtained and developed corresponding solution methods using real and complex representations of RB matrices.
	
	To assess the sensitivity and stability of the obtained solutions, we provided analytical bounds for the associated forward error. These theoretical results were validated through numerical examples, which demonstrated the effectiveness and robustness of the proposed RBTLSE approach compared to traditional RBLSE approach, especially in scenarios where both sides of the system are contaminated by noise.
	
	In future research, we aim to investigate purely imaginary and reduced biquaternion solutions to the RBTLSE problem. Additionally, we plan to explore the application of these special solutions in signal and image processing tasks.

\section*{Funding}
No funding was received for this study.

\section*{Conflict of Interest}
The author declare no potential conflict of interest.

\section*{Data Availability}
No data were generated or analyzed during this study; therefore, data sharing is not applicable.

\section*{ORCID}
Neha Bhadala \orcidC \href{https://orcid.org/0009-0001-9249-0611}{ \hspace{2mm}\textcolor{lightblue}{https://orcid.org/0009-0001-9249-0611}} \\

\bibliographystyle{abbrv}
\bibliography{Reference1}
\end{document}